\documentclass[11pt]{amsart}  
\usepackage{amsbsy,amsmath,amsthm,amssymb,color,verbatim}
\usepackage[backrefs]{amsrefs}
\usepackage{float}
\usepackage{fullpage,cancel,hyperref}
\usepackage{subcaption}
\usepackage{longtable}
\usepackage{url}
\usepackage{float}
\usepackage{wrapfig}
\newcommand{\g}{\mathfrak{g}}

\newcommand{\hroot}{\tilde\alpha}
\newcommand{\js}{\mathsf{js}}
\newcommand{\JS}{\mathrm{JS}}

\newcommand{\ljs}{\mathsf{ljs}}
\newcommand{\LJS}{\mathrm{LJS}}
\newcommand{\RJS}{\mathrm{RJS}}
\newcommand{\PJS}{\mathrm{PJS}}
\newcommand{\Cat}{\mathrm{Cat}}
\newcommand{\T}{\mathcal{T}}
\renewcommand{\wp}{K}

\newcommand{\e}{\varepsilon}

\newcommand*\circled[2]{\node[shape=circle,draw,inner sep=3pt,fill=black, scale= 0.35] at #1 (char) {\textcolor{white}{#2}};}
\newcommand*\dashcircled[2]{\node[shape=circle,draw,inner sep=3pt, scale= 0.35] at #1 (char) {#2};} 
\DeclareMathOperator{\SYT}{SYT}
\DeclareMathOperator{\perm}{perm}

\usepackage[dvipsnames]{xcolor}
\usepackage{graphicx}
\usepackage{subcaption}
\usepackage{xcolor}
\usepackage{multicol}
\usepackage{supertabular}
\usepackage{tikz}
\usepackage{tkz-graph}
\tikzstyle{vertex}=[circle, draw, inner sep=0pt, minimum size=4pt]

\tikzstyle{vtx}=[circle, draw, inner sep=0pt, minimum size=8pt]

\definecolor{darkgreen}{cmyk}{.9,0,.9,.2}
\definecolor{midgray}{gray}{0.60}
\definecolor{lightgray}{gray}{0.90}
\definecolor{lmgray}{gray}{0.70}
\usepackage{colortbl}
\def\a{\alpha}

\newcommand\multiset[2]%
{\mathchoice{\left(\kern-0.5em{\binom{#1}{#2}}\kern-0.5em\right)}
            {\bigl(\kern-0.3em{\binom{#1}{#2}}\kern-0.3em\bigr)}
            {\bigl(\kern-0.3em{\binom{#1}{#2}}\kern-0.3em\bigr)}
            {\bigl(\kern-0.3em{\binom{#1}{#2}}\kern-0.3em\bigr)}}

\newcommand{\ph}{\phantom{-}}
\allowdisplaybreaks

\def\a{\alpha}

\def\vol{\mathrm{vol}}

\newcommand{\bfa}{\mathbf{a}}
\newcommand{\bfb}{\mathbf{b}}

\newcommand{\bfj}{\mathbf{j}}

\newcommand{\bfs}{\mathbf{s}}
\newcommand{\bft}{\mathbf{t}}
\newcommand{\bfu}{\mathbf{u}}
\newcommand{\bfv}{\mathbf{v}}

\usepackage{multirow}
\usepackage{longtable}

\usepackage{makecell}
\setcellgapes{4pt}

\usepackage{relsize}
\makeatletter
\newtheorem*{rep@theorem}{\rep@title}
\newcommand{\newreptheorem}[2]{%
\newenvironment{rep#1}[1]{%
 \def\rep@title{#2 \ref{##1}}%
 \begin{rep@theorem}}%
 {\end{rep@theorem}}}
\makeatother

\makeatletter
\newtheorem*{rep@conjecture}{\rep@title}
\newcommand{\newrepconjecture}[2]{%
\newenvironment{rep#1}[1]{%
 \def\rep@title{#2 \ref{##1}}%
 \begin{rep@conjecture}}%
 {\end{rep@conjecture}}}
\makeatother

\makeatletter
\newcommand{\addresseshere}{%
  \enddoc@text\let\enddoc@text\relax
}
\makeatother

\usepackage{multicol}

\newtheorem{theorem}{Theorem}[section]
\newtheorem{lemma}[theorem]{Lemma}
\newtheorem{proposition}[theorem]{Proposition}
\newtheorem{corollary}[theorem]{Corollary}

\newtheorem{question}{Open Question}[section]

\newtheorem{problem}[theorem]{Problem}

\theoremstyle{definition}
  \newtheorem{remark}[theorem]{Remark}
  \newtheorem{definition}[theorem]{Definition}
  \newtheorem{example}[theorem]{Example}
  
  \newtheorem*{C1}{Theorem~\ref{thm:unresA}}
  
\numberwithin{equation}{section}

\title{Kostant's partition function and\\ magic multiplex juggling sequences}

\author[Benedetti]{Carolina Benedetti}
\address[C.\ Benedetti]{Departamento de Matem\'aticas\\Universidad de los Andes\\Bogot\'a\\Colombia} 
\email{\textcolor{blue}{\href{mailto:c.benedetti@uniandes.edu.co}{c.benedetti@uniandes.edu.co}}}
\urladdr{\url{https://sites.google.com/site/carobenedettimath/home}}

\author[Hanusa]{Christopher R.\ H.\ Hanusa}
\address[C.\ R.\ H.\ Hanusa]{Department of Mathematics \\ Queens College (CUNY)\\ 
Flushing, NY 11367\\ United States}
\email{\textcolor{blue}{\href{mailto:chanusa@qc.cuny.edu}{chanusa@qc.cuny.edu}}}
\urladdr{\url{http://qc.edu/~chanusa/}}

\author[Harris]{Pamela E.\ Harris}
\address[P.\ E.\ Harris]{Department of Mathematics and Statistics\\ Williams College\\
Williamstown, MA 01267, United States} 
\email{\textcolor{blue}{\href{mailto:peh2@williams.edu}{peh2@williams.edu}}}
\urladdr{\url{https://www.pamelaeharris.com}}

\author[Morales]{Alejandro H.\ Morales}
\address[A.\ H.\ Morales]{Department of Mathematics and Statistics, University of Massachusetts, Amherst, MA, 01003, United States} 
\email{\textcolor{blue}{\href{mailto:ahmorales@math.umass.edu}{ahmorales@math.umass.edu}}}
\urladdr{\url{http://people.math.umass.edu/~ahmorales/}}

\author[Simpson]{Anthony Simpson}
\address[A.\ Simpson]{Department of Mathematics and Statistics\\ Williams College\\
Williamstown, MA 01267, United States} 
\email{\textcolor{blue}{\href{mailto:als7@williams.edu}{als7@williams.edu}}}

\date{\today}


\begin{document}

\begin{abstract}
Kostant's partition function is a vector partition function that counts the number of ways one can express a weight of a Lie algebra $\mathfrak{g}$ as a nonnegative integral linear combination of the positive roots of $\mathfrak{g}$. Multiplex juggling sequences are generalizations of juggling sequences that specify an initial and terminal configuration of balls and allow for multiple balls at any particular discrete height. Magic multiplex juggling sequences generalize further to include magic balls, which cancel with standard balls when they meet at the same height. 

In this paper, we establish a combinatorial equivalence between positive roots of a Lie algebra and throws during a juggling sequence.  This provides a juggling framework to calculate Kostant's partition functions, and a partition function framework to compute the number of juggling sequences. From this equivalence we provide a broad range of consequences and applications connecting this work to polytopes, posets, positroids, and weight multiplicities. 
\end{abstract}

\maketitle


\section{Introduction}

Juggling was formalized as a mathematical concept around 1980~\cite{Shannon}.  Juggling sequences keep track of the movement of a finite number of balls by recording their discrete heights at discrete time steps, under the condition that only one ball is caught and thrown at any time \cite{BEGW}. This imitates the continuous action of how a juggler catches and throws balls while juggling. 

Since the formalization of mathematical juggling, researchers have determined the number of distinct juggling sequences satisfying certain juggling parameters. Moreover, a number of extensions to juggling sequences have been made, including Butler and Graham's recent extension to multiplex juggling sequences \cite{BG}, in which the juggler can now hold up to $m$ balls in their hand (their hand capacity).  Some of these sequences were enumerated in \cite{EMJP}.  To learn more about the mathematics of juggling, we recommend the short history given by Varpanen \cite{Varpanen12ashort} or the more comprehensive summary by Polster \cite{PolsterBook}.

The study of juggling patterns has been shown to have connections to many areas of mathematics such as walks on graphs \cites{CG,BG}, probability and Markov processes \cites{LV,ABCN}, $q$-analogues \cites{E,ER1994,ELV}, positroid varieties \cite{KLS}, vector compositions \cite{Stadler}, and vector partition functions \cites{HIW,HIO}. It is this last connection that motivates our current work.

We recall that a weight $\mu$ of a Lie algebra $\mathfrak{g}$ is a linear functional on the dual of a Cartan subalgebra of $\mathfrak{g}$. More simply, a weight is a vector in $\mathbb{R}^n$, for some appropriate $n$ and with certain defining conditions on its entries dependent on the Lie algebra of interest. In this way, the set of positive roots of $\mathfrak{g}$ is a finite linearly dependent set of vectors in $\mathbb{R}^n$, which we denote by $\Phi^+$. As for weights, the linear dependence relation between the vectors in $\Phi^+$ depends on the Lie algebra involved; we make this precise in Section~\ref{sec:lietypes}. Then Kostant's partition function is a vector partition function that counts the number of ways to express a weight $\mu$ of a Lie algebra as a sum of its positive roots. We denote this count by $K(\mu)$.

We extend the definition of a multiplex juggling sequence to deal with the negative integers that arise as entries in the vectors defining the positive roots of a Lie algebra.  This leads to the concept of a magic multiplex juggling sequence that includes both the standard non-magic balls as well as new magic balls. In these juggling sequences, a magic ball and a non-magic ball located at the same height nullify each other and both disappear.  Since all of our results involve magic multiplex juggling sequences, throughout this manuscript, we drop the language ``magic multiplex'' and refer to them simply as juggling sequences.
We let the integer tuple $\langle a_1,a_2, \cdots, a_k\rangle$ describe the configuration of $\sum_{i=1}^ka_i$ balls in which $a_i$ balls are at height $i$ (and where they are magic balls if $a_i<0$).

The key insight setting the foundation for our results is that there is a combinatorial equivalence between the throwing of a ball during a juggling sequence at time $i$ to height $j$ and the positive root $\e_i-\e_{i+j}$ appearing in a partition of a weight of a Lie algebra of type $A_r$, where  $\{\e_1,\ldots,\e_{r+1}\}$ are the standard basis vectors in $\mathbb{R}^{r+1}$. Thinking of a juggling sequence as a multiset of throws and a vector partition as a multiset of positive roots establishes a correspondence between these collective objects, which we develop in Sections~\ref{sec:bijection} and \ref{sec:restrict}.

The correspondence leads to a bijection (Theorem~\ref{thm:unresAA}) between the partitions of a weight $\mu=\sum_{i=1}^{r+1}\mu_{i}\e_i$ of type $A_r$ (where each $\mu_{i}\in\mathbb{Z}$) and juggling sequences of length $r$ where balls start with the configuration $\langle \mu_{1},\ldots,\mu_{r} \rangle$ and end with all $\mu_{1}+\cdots+\mu_r$ balls at height one, the cardinality of which is denoted $\js(\langle\mu_1,\ldots,\mu_{r}\rangle, \langle \mu_{1}+\cdots+\mu_r \rangle, r)$.  This gives the following enumerative result for Kostant's partition function $\wp(\mu)$.

\begin{C1}
Let $\mu$ be a weight of the Lie algebra of type $A_r$ and let $\langle\mu_1,\ldots,\mu_{r+1}\rangle$ be its standard basis vector representation. Then
\[\wp(\mu) = \js(\langle\mu_1,\ldots,\mu_{r}\rangle, \langle \mu_{1}+\cdots+\mu_r \rangle, r).\]
\end{C1}

The same correspondence provides a bijection (Theorem~\ref{thm:jugtopart}) between a general set of juggling sequences with initial configuration of balls $\bfa$, a terminal configuration of balls $\bfb$ (with both $\bfa$ and $\bfb$ containing the same net number of balls), hand capacity $m$, and length $n$ and the set of partitions of a weight $\delta$ of type $A_n$ (specified by $\bfa$, $\bfb$, and $n$) with an explicit restriction on the positive roots allowed in the partition.

These results are just the tip of the mathematical iceberg.  The Lie algebra of type $A$ is one in a family of four classical simple Lie algebras.  Hence, one could ask how the bijections between juggling sequences and partitions of weights generalize to other Lie types.  In Section~\ref{sec:othertypes-juggling} we define a new variety of juggling that models partitions of weights in Lie algebras of type $B$, $C$, and $D$.  Then in Section~\ref{sec:sb} we provide a way to count these new juggling sequences as sums of standard (type $A$) juggling sequences, which makes use of a result of Schmidt and Bincer \cite{SB}.  See Corollary \ref{cor:applyingSB}.

One motivation for this paper was the value of Kostant's partition function on the highest root of the Lie algebras of types $B$ and $C$.  Harris, Insko, and Omar \cite{HIO} gave generating functions for these values, verifying the claim of Harris, Insko, and Williams \cite{HIW} that the generating functions for the value of Kostant's partition function on the highest root were the same as those for counting certain multiplex juggling sequences.  Although the generating functions agreed, neither set of authors gave combinatorial proofs for these results.  Section~\ref{sec:HIOsubsec} is devoted to applying our methods to establish these bijective proofs, which are stated in Theorems~\ref{thm:solution-problem 1} and \ref{thm:solution-problem 2}.

In Section~\ref{sec:connections}, we give further applications and connections from the correspondence between partitions of a weight of a Lie algebra of type $A$ and juggling sequences. Since the former can be seen as lattice points of flow polytopes  \cite{BB} or elements of a poset \cite{JO}, we define the equivalent juggling polytope (see Figure~\ref{fig:juggling polytope}) of  real-valued juggling sequences and juggling poset (see Figure~\ref{fig:juggling poset}). The connection with polytopes allows us to conclude polynomiality properties of the number of juggling sequences (Corollary~\ref{cor:polynomiality} and Corollary~\ref{cor:lidskii-juggling}). In addition, we translate results from Chung--Graham \cite{CG} and Butler--Graham \cite{BG} on juggling sequences to give permanent and determinant formulas for values of restricted Kostant's partition function on the highest root of the Lie algebra of type $A$ (Theorem~\ref{prop:juggling-perm-det}).

This manuscript is organized as follows.
In Section~\ref{sec:background}, we give precise mathematical definitions of juggling and of the Lie algebras of interest. Section~\ref{sec:bijection} provides the bijection between partitions of a weight of a Lie algebra of type $A$ and certain multiplex juggling sequences, while Section~\ref{sec:restrict} shows how to convert a multiplex juggling sequence into a partition of a weight of a Lie algebra of type $A$ with certain restrictions.  Section~\ref{sec:HIO} extends multiplex juggling sequences to other classical Lie types as described above.  Section~\ref{sec:connections} presents applications and connections to polytopes, posets, generating functions, positroids, and weight multiplicities, with open problems scattered throughout. We end the manuscript by providing a dictionary of notation for the benefit of the authors and the reader.

\section{Background}\label{sec:background}

\subsection{Juggling sequences}
We now present the mathematics of juggling using the definitions and notation presented by Butler and Graham \cite{BG}.  We consider time to be broken down into discrete steps.  When we juggle, we throw each ball into the air immediately upon catching it---the ball returns to our hand after a fixed amount of time that is determined by how high the ball was thrown. 
A {\em juggling state} records the position of the  $b\in\mathbb{N}$ indistinguishable balls at a given time. That is, a juggling state is a vector $\bfs = \langle s_1, \ldots, s_h \rangle$ of nonnegative integers that sum to $b$, where $s_i$ denotes the number of balls at height $i$ with $i=1,2,\ldots, h$. We say that the {\em height} $h$ of $\bfs$ is its largest non-zero index. 

A {\em multiplex juggling sequence} $S=(\bfs_0, \bfs_1, \ldots, \bfs_n )$ represents the act of juggling over time. Hence, two successive juggling states $\bfs_{i-1}=\langle s_1, \ldots, s_h \rangle$ and $\bfs_i$ must satisfy 
\begin{equation}
\label{eq:successive}
    \bfs_i=\langle s_2 + b_1, s_3 + b_2, \ldots, s_h + b_{h-1}, b_h, \ldots, b_{h'} \rangle,
\end{equation}
where the nonnegative integers $b_j$ satisfy $\sum_{j=1}^{h'} b_j = s_1$.  This defining condition represents accounting for the throwing of the $s_1$ balls that land in your hand at time $i$ to heights $j=1,\ldots,h'$ (where $b_j$ balls return to your hand at time $i+j$) and gravity (balls in the air are one time step closer to landing in your hand).  When a ball at height $1$ at time $i-1$ is thrown to height $j$ at time $i$, we call this {\em a throw at time $i$ to height $j$}, which we denote by $T_{i,j}$. Since we only consider multiplex juggling sequences we  drop the word multiplex.

A juggling sequence is said to be \emph{periodic} if its initial state is the same as its terminal state. 

A juggling sequence can also have a {\em hand capacity constraint}, which specifies the maximum number of balls at any height. Equivalently, any entry in a juggling state must be less than or equal to $m$, representing that we can only catch $m$ balls at any given time. When $m=1$ we recover standard juggling sequences. When $m$ is greater than or equal to the number of balls, the hand capacity parameter does not play a role and can be ignored.

We define $\JS(\bfa,\bfb,n,m)$ to be the set of all juggling sequences of length $n$, hand capacity $m$, initial state $\bfa$, and terminal state  $\bfb$, and we let $\js(\bfa,\bfb,n,m)$ denote the number of juggling sequences with these parameters. When there is no hand capacity (or if $m\geq b$), we write $\JS(\bfa,\bfb,n)$ and we let $\js(\bfa,\bfb,n)$ denote its cardinality. 

In physical juggling, the balls rise and fall in a parabolic motion, but it is easier to represent a juggling state $\bfs=\langle s_1, \ldots, s_h \rangle$ visually as a conveyor in which $s_1,s_2,\ldots,s_h$ balls are located in buckets at heights $1,2,\ldots,h$, respectively.  A juggling sequence is therefore a sequence of juggling states in which the balls at heights $2$ through $h$ descend by one unit height in successive time steps, while the balls at height $1$ are redistributed into buckets of some desired heights.  We illustrate the conveyor with the following example.

\begin{example}
Figure \ref{fig:multjugseq} shows the juggling sequence $(
\langle1, 1\rangle,
\langle2\rangle,
\langle0, 1, 1\rangle,
\langle1, 1\rangle
)$, which is a member of  $\JS(\langle 1, 1 \rangle, \langle 1, 1 \rangle, 3, 2)$.  Note that there are two times at which balls are thrown. At time $t=1$, a ball is thrown to height $1$ and at time $t=2$, the two balls are thrown to heights $2$ and $3$ respectively.

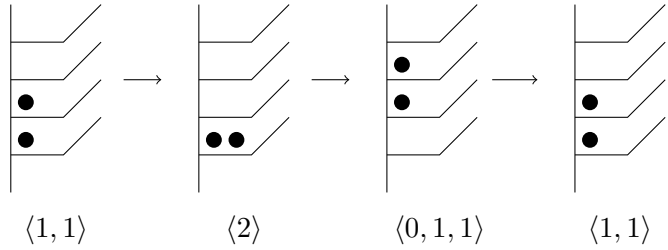
\begin{figure}[tbhp]
\centering
\begin{tikzpicture}
\draw (0,-1.5) -- (0,1.0);
\draw (0, -1.0) -- (0.7, -1.0) -- (1.2, -0.5);
\fill (0.2,-0.8) circle (3pt);
\draw (0, -0.5) -- (0.7, -0.5) -- (1.2, 0.0);
\fill (0.2,-0.3) circle (3pt);
\draw (0, 0.0) -- (0.7, 0.0) -- (1.2, 0.5);
\draw (0, 0.5) -- (0.7, 0.5) -- (1.2, 1.0);
\draw (0.6, -2.0) node {$\langle1, 1\rangle$};
\draw[->] (1.5,0) -- (2.0,0);
\draw (2.5,-1.5) -- (2.5,1.0);
\draw (2.5, -1.0) -- (3.2, -1.0) -- (3.7, -0.5);
\fill (2.7,-0.8) circle (3pt);
\fill (3.0,-0.8) circle (3pt);
\draw (2.5, -0.5) -- (3.2, -0.5) -- (3.7, 0.0);
\draw (2.5, 0.0) -- (3.2, 0.0) -- (3.7, 0.5);
\draw (2.5, 0.5) -- (3.2, 0.5) -- (3.7, 1.0);
\draw (3.1, -2.0) node {$\langle2\rangle$};
\draw[->] (4.0,0) -- (4.5,0);
\draw (5.0,-1.5) -- (5.0,1.0);
\draw (5.0, -1.0) -- (5.7, -1.0) -- (6.2, -0.5);
\draw (5.0, -0.5) -- (5.7, -0.5) -- (6.2, 0.0);
\fill (5.2,-0.3) circle (3pt);
\draw (5.0, 0.0) -- (5.7, 0.0) -- (6.2, 0.5);
\fill (5.2,0.2) circle (3pt);
\draw (5.0, 0.5) -- (5.7, 0.5) -- (6.2, 1.0);
\draw (5.7, -2.0) node {$\langle0, 1, 1\rangle$};
\draw[->] (6.4,0) -- (7.0,0);
\draw (7.5,-1.5) -- (7.5,1.0);
\draw (7.5, -1.0) -- (8.2, -1.0) -- (8.7, -0.5);
\fill (7.7,-0.8) circle (3pt);
\draw (7.5, -0.5) -- (8.2, -0.5) -- (8.7, 0.0);
\fill (7.7,-0.3) circle (3pt);
\draw (7.5, 0.0) -- (8.2, 0.0) -- (8.7, 0.5);
\draw (7.5, 0.5) -- (8.2, 0.5) -- (8.7, 1.0);
\draw (8.1, -2.0) node {$\langle1, 1\rangle$};
\end{tikzpicture}
\caption{The  juggling sequence $(
\langle1, 1\rangle,
\langle2\rangle,
\langle0, 1, 1\rangle,
\langle1, 1\rangle
)$.}
\label{fig:multjugseq}
\end{figure}
\end{example}

\subsection{Magic juggling sequences} 

A magic juggling sequence removes the restriction that all entries in a juggling state must be nonnegative integers. An entry $s_j<0$ in a juggling state denotes $|s_j|$ ``magic'' balls at height $j$.  The name magic is used because when a magic ball and a standard non-magic ball come into contact they ``magically'' disappear.  In this way, a juggling state is now an integer vector $\bfs = \langle s_1, \ldots, s_h \rangle$, and a magic juggling sequence $S=(\bfs_0,\ldots,\bfs_n)$ satisfies Equation~\eqref{eq:successive} and requires that any magic balls not be redistributed.  For this reason, any hand capacity constraint only applies to non-magic balls.  See Figure~\ref{fig:magicmultjuggseq} for an example of a magic  juggling sequence in $\JS(\langle 2, 1, 0, -1 \rangle, \langle 1, 1 \rangle, 3)$.   Magic balls arising from negative entries in the starting state are represented by white balls. 

Because a magic juggling sequence is simply a standard juggling sequence when the entries of $\bfa$ are nonnegative, we can extend the definitions of $\JS(\bfa,\bfb,n,m)$ and $\js(\bfa,\bfb,n,m)$ to represent the set and number of magic juggling sequences, respectively.  We once again drop the last parameter if there is no hand capacity constraint or if the hand capacity constraint is larger than the number of non-magic balls.  Since we only consider magic juggling sequences we drop the word magic.

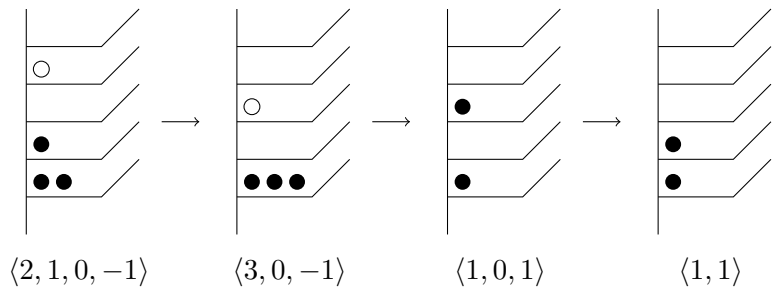
\begin{figure}[htbp]
    \centering
    \begin{tikzpicture}
\draw (0,-1.5) -- (0,1.5);
\draw (0, -1.0) -- (1.0, -1.0) -- (1.5, -0.5);
\fill (0.2,-0.8) circle (3pt);
\fill (0.5,-0.8) circle (3pt);
\draw (0, -0.5) -- (1.0, -0.5) -- (1.5, 0.0);
\fill (0.2,-0.3) circle (3pt);
\draw (0, 0.0) -- (1.0, 0.0) -- (1.5, 0.5);
\draw (0, 0.5) -- (1.0, 0.5) -- (1.5, 1.0);
\draw (0.2,0.7) circle (3pt);
\draw (0, 1.0) -- (1.0, 1.0) -- (1.5, 1.5);
\draw (0.7, -2.0) node {$\langle2, 1, 0, -1\rangle$};
\draw[->] (1.8,0) -- (2.3,0);
\draw (2.8,-1.5) -- (2.8,1.5);
\draw (2.8, -1.0) -- (3.8, -1.0) -- (4.3, -0.5);
\fill (3.0,-0.8) circle (3pt);
\fill (3.3,-0.8) circle (3pt);
\fill (3.6,-0.8) circle (3pt);
\draw (2.8, -0.5) -- (3.8, -0.5) -- (4.3, 0.0);
\draw (2.8, 0.0) -- (3.8, 0.0) -- (4.3, 0.5);
\draw (3.0,0.2) circle (3pt);
\draw (2.8, 0.5) -- (3.8, 0.5) -- (4.3, 1.0);
\draw (2.8, 1.0) -- (3.8, 1.0) -- (4.3, 1.5);
\draw (3.5, -2.0) node {$\langle3, 0, -1\rangle$};
\draw[->] (4.6,0) -- (5.1,0);
\draw (5.6,-1.5) -- (5.6,1.5);
\draw (5.6, -1.0) -- (6.6, -1.0) -- (7.1, -0.5);
\fill (5.8,-0.8) circle (3pt);
\draw (5.6, -0.5) -- (6.6, -0.5) -- (7.1, 0.0);
\draw (5.6, 0.0) -- (6.6, 0.0) -- (7.1, 0.5);
\fill (5.8,0.2) circle (3pt);
\draw (5.6, 0.5) -- (6.6, 0.5) -- (7.1, 1.0);
\draw (5.6, 1.0) -- (6.6, 1.0) -- (7.1, 1.5);
\draw (6.3, -2.0) node {$\langle1, 0, 1\rangle$};
\draw[->] (7.4,0) -- (7.9,0);
\draw (8.4,-1.5) -- (8.4,1.5);
\draw (8.4, -1.0) -- (9.4, -1.0) -- (9.9, -0.5);
\fill (8.6,-0.8) circle (3pt);
\draw (8.4, -0.5) -- (9.4, -0.5) -- (9.9, 0.0);
\fill (8.6,-0.3) circle (3pt);
\draw (8.4, 0.0) -- (9.4, 0.0) -- (9.9, 0.5);
\draw (8.4, 0.5) -- (9.4, 0.5) -- (9.9, 1.0);
\draw (8.4, 1.0) -- (9.4, 1.0) -- (9.9, 1.5);
\draw (9.1, -2.0) node {$\langle1, 1\rangle$};
\end{tikzpicture}
\caption{Example of a magic  juggling sequence.}
    \label{fig:magicmultjuggseq}
\end{figure}

\subsection{Labeled juggling sequences}

A {\em labeled juggling sequence} allows the balls to be distinguishable. In this case we let $l\leq b$ denote the labels we may use. In this case, a juggling state is a vector $\bfs = \langle \bfu_1, \ldots, \bfu_h \rangle$ of integer tuples  $\bfu_i=[u_i^1,\ldots,u_i^l]\in\mathbb{Z}^l$ where $|u_i^j|$ records how many balls of label $j$ are present at height $i$. Note that if $u_i^j>0$ then there are $|u_i^j|$ balls labeled  $j$ at height $i$, and if $u_i^j<0$ then there are $|u_i^j|$ magic balls labeled $j$ at height $i$.

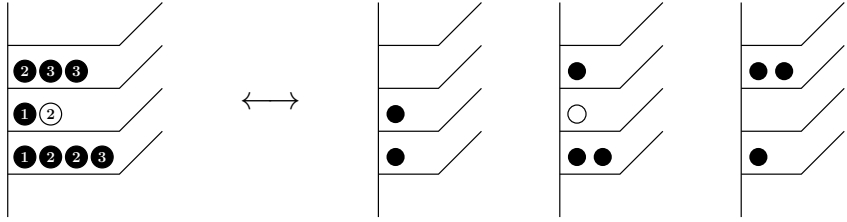
\begin{figure}[htbp]
    \centering
    \resizebox{4.5in}{!}{
    \begin{tikzpicture}
    \draw (0,-1.5) -- (0,1.0);
    \draw (0, -1.0) -- (1.3, -1.0) -- (1.8, -0.5);
    \circled{(0.2,-0.8)}{\Large\textbf{1}}
    \circled{(0.5,-0.8)}{\Large\textbf{2}}
    \circled{(0.8,-0.8)}{\Large\textbf{2}}
    \circled{(1.1,-0.8)}{\Large\textbf{3}}
    \draw (0, -0.5) -- (1.3, -0.5) -- (1.8, 0.0);
    \circled{(0.2,-0.3)}{\Large\textbf{1}}
    \dashcircled{(0.5,-0.3)}{\Large\textbf{2}}
    \draw (0, 0.0) -- (1.3, 0.0) -- (1.8, 0.5);
    \circled{(0.2,0.2)}{\Large\textbf{2}}
    \circled{(0.5,0.2)}{\Large\textbf{3}}
    \circled{(0.8,0.2)}{\Large\textbf{3}}
    \draw (0, 0.5) -- (1.3, 0.5) -- (1.8, 1.0);
    \end{tikzpicture}
    \qquad
    \raisebox{.5in}{$\longleftrightarrow$}
    \qquad
        \begin{tikzpicture}
    \draw (0,-1.5) -- (0,1.0);
    \draw (0, -1.0) -- (0.7, -1.0) -- (1.2, -0.5);
    \fill (0.2,-0.8) circle (3pt);
    \draw (0, -0.5) -- (0.7, -0.5) -- (1.2, 0.0);
    \fill (0.2,-0.3) circle (3pt);
    \draw (0, 0.0) -- (0.7, 0.0) -- (1.2, 0.5);
    \draw (0, 0.5) -- (0.7, 0.5) -- (1.2, 1.0);
    \end{tikzpicture}
\qquad
    \begin{tikzpicture}
    \draw (2.5,-1.5) -- (2.5,1.0);
    \draw (2.5, -1.0) -- (3.2, -1.0) -- (3.7, -0.5);
    \fill (2.7,-0.8) circle (3pt);
    \fill (3.0,-0.8) circle (3pt);
    \draw (2.5, -0.5) -- (3.2, -0.5) -- (3.7, 0.0);
    \draw (2.7,-0.3) circle (3pt);
    \draw (2.5, 0.0) -- (3.2, 0.0) -- (3.7, 0.5);
    \fill (2.7,0.2) circle (3pt);
    \draw (2.5, 0.5) -- (3.2, 0.5) -- (3.7, 1.0);
    \end{tikzpicture}
\qquad
    \begin{tikzpicture}
    \draw (5.0,-1.5) -- (5.0,1.0);
    \draw (5.0, -1.0) -- (5.7, -1.0) -- (6.2, -0.5);
    \fill (5.2,-0.8) circle (3pt);
    \draw (5.0, -0.5) -- (5.7, -0.5) -- (6.2, 0.0);
    \draw (5.0, 0.0) -- (5.7, 0.0) -- (6.2, 0.5);
    \fill (5.2,0.2) circle (3pt);
    \fill (5.5,0.2) circle (3pt);
    \draw (5.0, 0.5) -- (5.7, 0.5) -- (6.2, 1.0);
    \end{tikzpicture}
    }
    \caption{The labeled juggling state $\langle[1,2,1], [1,-1,0], [0, 1, 2]\rangle$ decomposes into three juggling states,    $\bfs^1=\langle1,1\rangle$,
    $\bfs^2=\langle2,-1,1\rangle$, and
    $\bfs^3=\langle1,0,2\rangle$. 
    }
    \label{fig:labeled_state}
\end{figure}

There is a natural decomposition of a labeled juggling state with $l$ labels into $l$ juggling states, each independently corresponding to having balls with only one specified label.   If $\bfs = \langle \bfu_1, \ldots, \bfu_h \rangle$ is a labeled juggling state, then $\bfs^j = \langle u_1^j, \ldots, u_h^j \rangle$ is an unlabeled juggling state, as in Figure~\ref{fig:labeled_state}.
Similarly, a labeled juggling sequence $S=\langle\bfs_0,\ldots,\bfs_n\rangle$ involving $l$ labels  can be decomposed into $l$ juggling sequences $S^j=\langle\bfs_0^j,\ldots,\bfs_n^j\rangle$, one for each label $1\leq j\leq l$.  An example of a labeled juggling sequence is shown in Figure~\ref{fig:labeled_sequence}.

\begin{figure}[htbp]
    \centering
    \resizebox{0.7\textwidth}{!}{
    \begin{tikzpicture}
    \draw (0,-1.5) -- (0,1.0);
    \draw (0, -1.0) -- (1.3, -1.0) -- (1.8, -0.5);
    \circled{(0.2,-0.8)}{\bf \Large 1}
    \circled{(0.5,-0.8)}{\bf \Large 2}
    \circled{(0.8,-0.8)}{\bf \Large 2}
    \circled{(1.1,-0.8)}{\bf \Large 3}
    \draw (0, -0.5) -- (1.3, -0.5) -- (1.8, 0.0);
    \circled{(0.2,-0.3)}{\bf \Large 1}
    \dashcircled{(0.5,-0.3)}{\bf \Large 2}
    \draw (0, 0.0) -- (1.3, 0.0) -- (1.8, 0.5);
    \circled{(0.2,0.2)}{\bf \Large 2}
    \circled{(0.5,0.2)}{\bf \Large 3}
    \circled{(0.8,0.2)}{\bf \Large 3}
    \draw (0, 0.5) -- (1.3, 0.5) -- (1.8, 1.0);
    \draw[->] (2.0,0) -- (2.5,0);
    \draw (3.0,-1.5) -- (3.0,1.0);
    \draw (3.0, -1.0) -- (4.3, -1.0) -- (4.8, -0.5);
    \circled{(3.2,-0.8)}{\bf \Large 1}
    \circled{(3.5,-0.8)}{\bf \Large 3}
    \draw (3.0, -0.5) -- (4.3, -0.5) -- (4.8, 0.0);
    \circled{(3.2,-0.3)}{\bf \Large 2}
    \circled{(3.5,-0.3)}{\bf \Large 3}
    \circled{(3.8,-0.3)}{\bf \Large 3}
    \draw (3.0, 0.0) -- (4.3, 0.0) -- (4.8, 0.5);
    \circled{(3.2,0.2)}{\bf \Large 1}
    \circled{(3.5,0.2)}{\bf \Large 2}
    \draw (3.0, 0.5) -- (4.3, 0.5) -- (4.8, 1.0);
    \draw[->] (5.0,0) -- (5.5,0);
    \draw (6.0,-1.5) -- (6.0,1.0);
    \draw (6.0, -1.0) -- (7.3, -1.0) -- (7.8, -0.5);
    \circled{(6.2,-0.8)}{\bf \Large 2}
    \circled{(6.5,-0.8)}{\bf \Large 3}
    \circled{(6.8,-0.8)}{\bf \Large 3}
    \circled{(7.1,-0.8)}{\bf \Large 3}
    \draw (6.0, -0.5) -- (7.3, -0.5) -- (7.8, 0.0);
    \circled{(6.2,-0.3)}{\bf \Large 1}
    \circled{(6.5,-0.3)}{\bf \Large 2}
    \draw (6.0, 0.0) -- (7.3, 0.0) -- (7.8, 0.5);
    \circled{(6.2,0.2)}{\bf \Large 1}
    \draw (6.0, 0.5) -- (7.3, 0.5) -- (7.8, 1.0);
    \draw[->] (8.0,0) -- (8.5,0);
    \draw (9.0,-1.5) -- (9.0,1.0);
    \draw (9.0, -1.0) -- (10.3, -1.0) -- (10.8, -0.5);
    \circled{(9.2,-0.8)}{\bf \Large 1}
    \circled{(9.5,-0.8)}{\bf \Large 2}
    \circled{(9.8,-0.8)}{\bf \Large 2}
    \draw (9.0, -0.5) -- (10.3, -0.5) -- (10.8, 0.0);
    \circled{(9.2,-0.3)}{\bf \Large 1}
    \circled{(9.5,-0.3)}{\bf \Large 3}
    \circled{(9.8,-0.3)}{\bf \Large 3}
    \circled{(10.1,-0.3)}{\bf \Large 3}
    \draw (9.0, 0.0) -- (10.3, 0.0) -- (10.8, 0.5);
    \draw (9.0, 0.5) -- (10.3, 0.5) -- (10.8, 1.0);
    \end{tikzpicture}}
    \caption{The labeled juggling sequence 
     $(\langle [1,2,1], [1,-1,0], [0,1,2] \rangle$,\newline $\langle [1,0,1], [0,1,2], [1,1,0] \rangle$, $\langle [0,1,3], [1,1,0], [1,0,0] \rangle$, $\langle [1,2,0], [1,0,3] \rangle)$.}
    \label{fig:labeled_sequence}
\end{figure}
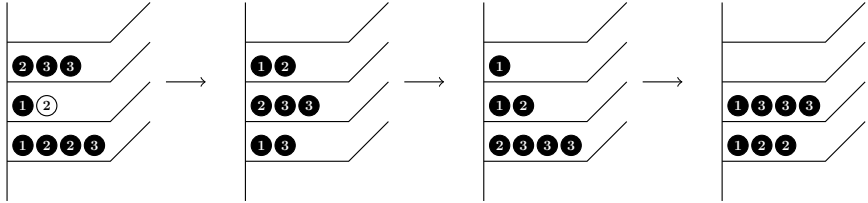

A consequence of this decomposition is a formula for the number $\ljs$ of labeled juggling sequences as a product of numbers of juggling sequences.

\begin{proposition}
\label{prop:labeledmagic}
Let $\bfa$ and $\bfb$ be labeled juggling states and $n$ be an integer. Then\[\ljs(\bfa,\bfb,n) = \prod_{j=1}^l \js(\bfa^j,\bfb^j,n).\]
\end{proposition}

Proposition~\ref{prop:labeledmagic} answers the question of Butler and Graham in \cite[Section~4.2]{BG} of how many juggling sequences exist when the balls are not identical, so as long as the hand capacity is large enough. Note that Proposition~\ref{prop:labeledmagic} does not apply when hand capacities are restrictive because a hand capacity for balls of all labels cannot be independently decomposed into multiple hand capacities.

\subsection{Kostant's partition function}\label{sec:lietypes}
In what follows we use the choices of vector space bases presented in Goodman and Wallach \cite[Section 2.4.3]{GW}.  For each Lie type, we describe a choice of simple roots, list the positive roots, and specify the corresponding highest root. Throughout, $\varepsilon_1,\varepsilon_2,\ldots,\varepsilon_d$ denote the standard basis vectors of $\mathbb{R}^d$, with $d$ being appropriately chosen depending on the Lie algebra. Also, $\Phi$ denotes a root system for a Lie algebra of type $A_r$, $B_r$, $C_r$ or $D_r$ and $\Delta = \{\alpha_1,\alpha_2,\ldots,\alpha_r\} \subset \Phi$ is a set of simple roots if every root in $\Phi$ can be written as a linear combination of the elements in $\Delta$ where the coefficients all have the same sign. We denote by $\Phi^+$ the set of \emph{positive roots of $\Phi$}, that is, the set of roots that can be written as linear combination of the roots in $\Delta$ with coefficients equal to 0 or 1. When multiple root systems are involved, we include a subscript in $\Phi^+$ to denote the Lie algebra of interest, as in Example \ref{ex:A3}.

\medskip
\noindent Type $A_r$ ($\mathfrak{sl}_{r+1}(\mathbb C)$): Let  $r\geq 1$ and let $\alpha_i=\varepsilon_i-\varepsilon_{i+1}$ for $1\leq i\leq r$. Then $\Delta = \{\alpha_i \ | \ 1 \leq i \leq r \}$ is a set of simple roots. The associated set of positive roots is $\Phi^+=\{\varepsilon_i-\varepsilon_j:1\leq i<j\leq r+1 \}$, and the highest root is $\tilde{\alpha}=\alpha_1+\alpha_2+\cdots+\alpha_r$.
 
\bigskip
\noindent
Type $B_r$ ($\mathfrak{so}_{2r+1}(\mathbb C)$): Let $r\geq 2$ and let $\alpha_i=\varepsilon_i-\varepsilon_{i+1}$ for $1\leq i\leq r-1$ and $\alpha_r=\varepsilon_r$. Then $\Delta = \{\alpha_i \ | \ 1 \leq i \leq r \}$ is a set of simple roots. The associated set of positive roots is $\Phi^+=\{\varepsilon_i-\varepsilon_j,\varepsilon_i+\varepsilon_j:1\leq i<j\leq r \}\cup\{\varepsilon_i:1\leq i\leq r\}$, and the highest root is $\tilde{\alpha}=\alpha_1+2\alpha_2+\cdots+2\alpha_r$.

\bigskip
\noindent
Type $C_r$ ($\mathfrak{sp}_{2r}(\mathbb C)$): Let $r\geq 3$ and let $\alpha_i=\varepsilon_i-\varepsilon_{i+1}$ for $1\leq i\leq r-1$ and $\alpha_r=2\varepsilon_r$. Then $\Delta = \{\alpha_i \ | \ 1 \leq i \leq r \}$ is a set of simple roots. The associated set of positive roots is $\Phi^+=\{\varepsilon_i-\varepsilon_j,\varepsilon_i+\varepsilon_j:1\leq i<j\leq r \}\cup\{2\varepsilon_i:1\leq i\leq r\}$, and the highest root is $\tilde{\alpha}=2\alpha_1+2\alpha_2+\cdots+2\alpha_{r-1}+\alpha_r$.

\bigskip
\noindent
Type $D_r$ ($\mathfrak{so}_{2r}(\mathbb C)$): Let $r\geq 4$ and let $\alpha_i=\varepsilon_i-\varepsilon_{i+1}$ for $1\leq i\leq r-1$ and $\alpha_r=\varepsilon_{r-1}+\varepsilon_r$. Then $\Delta = \{\alpha_i \ | \ 1 \leq i \leq r \}$ is a set of simple roots. The associated set of positive roots is $\Phi^+=\{\varepsilon_i-\varepsilon_j,\;\varepsilon_i+\varepsilon_j\ | \ 1\leq i<j\leq r\}$, and the highest root is $\tilde{\alpha}= \varepsilon_1 + \varepsilon_2=\alpha_1+2\alpha_2+\cdots+2\alpha_{r-2}+\alpha_{r-1}+\alpha_r$.

\begin{remark} \label{rem:roots-different-types}
Since a Lie algebra of type $A_r$ is a subalgebra of a Lie algebra $\g$ of Lie type $B_r$, $C_r$, or $D_r$, we simplify the exposition and reference positive roots in type $A_r$ as positive roots in other Lie types. More rigorously, 
we map positive roots in $\Phi_{A_r}^+$ to their analogous positive roots in $\Phi_{\g}^+$ by mapping every simple root $\a_i\in\Phi_{A_r}^+$ to the simple root $\a_i\in\Phi_{\g}^+$ for all $1\leq i\leq r$, and extend this map linearly to all positive roots in $\Phi_{A_r}^+$. 
\end{remark}

We are now ready to define our main object of study.
\begin{definition}
Kostant's partition function is a nonnegative integer valued function which counts the number of ways that a weight $\mu$ can be written as a nonnegative integer linear combination of elements of $\Phi^+$. We denote this count by $\wp(\mu)$. 
\end{definition}
When multiple root systems are involved, we include a subscript on $K$ to denote the Lie algebra of interest. 

An equivalent formulation is to represent a nonnegative integral linear combination of positive roots as a multiset $p=\{\beta_1,\beta_2,\ldots,\beta_l\}$ with the roots $\beta_i\in\Phi^+$.  We say $p$ is a \emph{partition of $\mu$} if $\sum_{i=1}^l\beta_i=\mu$ and we let $P(\mu)$ denote the set of partitions of the weight $\mu$.  We therefore have $K(\mu)=|P(\mu)|$.  By convention, $K(0)$ is defined to be $1$.

\begin{example}\label{ex:A3}
Let $\mu = \alpha_1 + 2\alpha_2 + \alpha_3$ be a weight of the Lie algebra of type $A_3$. Then
\begin{align*}
    P(\mu)&=\begin{Bmatrix}\{\alpha_1,\alpha_2,\alpha_2,\alpha_3\},\{\alpha_1 + \alpha_2, \alpha_2+\alpha_3\},\\\{\alpha_1 + \alpha_2,\alpha_2,\alpha_3\},\{\alpha_1,\alpha_2,\alpha_2 + \alpha_3\},\{\alpha_1 + \alpha_2 + \alpha_3,\alpha_2\}\end{Bmatrix}
\end{align*}
are all of the partitions of $\mu$ using the elements of $\Phi_{A_3}^+$, so $\wp(\mu) = 5$. 
\end{example}

In Section \ref{sec:restrict}, we will restrict the set of positive roots allowed in the partitions of a weight; we define a restricted Kostant's partition function as follows.

\begin{definition}
Let $\Lambda\subseteq \Phi^+$. Then $K_\Lambda$ defines a restriction to Kostant's partition function, counting the number of ways that a weight $\mu$ can be written as a nonnegative integer linear combination of elements of $\Lambda$. We let $P_\Lambda(\mu)$ denote the set of partitions of the weight $\mu$ using the positive roots in $\Lambda$, so $K_\Lambda(\mu)=|P_\Lambda(\mu)|$. 
\end{definition}

In the sections that follow it will be useful to express a weight $\mu$ in terms of the standard basis vectors. Hence if $\mu=\mu_1\e_1+\mu_2\e_2+\cdots+\mu_{r}\e_{r}$, 
where $\mu_i$ are real numbers for all $1\leq i\leq r$, 
we use the notation 
$\langle\mu_1,\mu_2,\ldots,\mu_r\rangle$ to represent the standard basis representation of $\mu$. 

\section{The bijection} \label{sec:bijection}

In this section we describe the bijection between juggling sequences and the partitions of a weight counted by Kostant's partition function. 
The key insight in this bijection is an association between throws made at time $i$ to height $j$ and the positive root $\e_i-\e_{i+j}$ in a Lie algebra of type $A_r$, with $r$ sufficiently large. In the analogous way, given a positive root $\e_i-\e_{i+j}$ of the Lie algebra of type $A_r$ we can associate it to the throw of a ball at time $i$ to height $j$ in a juggling sequence of an appropriate length. In the definitions that follow, we make this correspondence precise. 

\begin{definition}\label{def:T} For $r\geq 1$, define $\T_r$ to be the set of throws at time $i$ to height $j$ for integers $i$ and $j$ that satisfy $1\leq i<i+j\leq r+1$. Namely,
\[\T_r=\{T_{i,j}:1\leq i<i+j\leq r+1\}.\]
\end{definition}

\begin{proposition}
The function $\Gamma:\Phi_{A_r}^+\to\T_r$ defined by $\Gamma(\e_i-\e_{i+j})=T_{i,j}$ is a bijection.
\end{proposition}

\begin{proof}
The function is well defined because every positive root in type $A_r$ is of the form $\e_i-\e_{i+j}$ with $1\leq i<i+j\leq r+1$. By construction $\Gamma$ is injective. The surjectivity of $\Gamma$ follows because for every throw $T_{i,j}\in\T_r$, the positive root $\e_i-\e_{i+j}\in \Phi_{A_{r}}^+$ satisfies  $\Gamma(\e_i-\e_{i+j})=T_{i,j}$.
\end{proof}

We can characterize a juggling sequence through the multiset of throws that occur during the sequence. 
Given that $\Gamma$ is a bijection between $\Phi_{A_{r}}^+$ and $\T_r$, then so is its extension to multisets.

\begin{definition}\label{def:gamma}
We extend $\Gamma$ to apply to a multiset $p$ of positive roots $\Phi_{A_r}^+$.  Define $\Gamma(p)$ to be the multiset $\{\Gamma(\beta): \mbox{$\beta\in p$}\}\subseteq \T_r$. 
\end{definition}

Next we give a connection between the weight we partition and the multiset of throws in a juggling sequence.

\begin{definition}
Consider a juggling sequence $S = \langle \bfs_0, \bfs_1, \ldots, \bfs_n\rangle$ whose corresponding multiset of throws is $\T$.
Define the {\em net change vector} $\delta(S)$ of $S$ as 
\[ 
\delta(S) = \sum_{T\in\T}\Gamma^{-1}(T)=\sum_{i=1}^n\sum_{j\in J_i} (\e_i - \e_{i+j}),
\] 
where $J_i$ is the multiset of heights to which balls are thrown at time $i$.
\end{definition}

For example, the net change vector of the juggling sequence $S=(
\langle1, 1\rangle,
\langle2\rangle,
\langle0, 1, 1\rangle,
\langle1, 1\rangle
)$, as illustrated in Figure~\ref{fig:multjugseq}, is 
\[
\delta(S)=(\e_1-\e_2)+(\e_2-\e_4)+(\e_2-\e_5)=\e_1+\e_2-\e_4-\e_5,
\]
while the net change vector of the juggling sequence
$S=(\langle2,1,0,-1\rangle,
\langle3,0,-1\rangle,
\langle1,0,1\rangle,
\langle1,1\rangle)$ from Figure~\ref{fig:magicmultjuggseq} is
\[
\delta(S)=2(\e_1-\e_2)+(\e_2-\e_3)+(\e_2-\e_4)+(\e_2-\e_5)+(\e_3-\e_4)=2\e_1+\e_2-2\e_4-\e_5.
\]

\begin{theorem}\label{thm:netheight}
 Consider initial state $\bfa=\langle a_1,\ldots,a_s\rangle$, terminal state $\bfb=\langle b_1,\ldots,b_t\rangle$, sequence length $n$, and hand capacity $m$. If $a_1+\cdots+a_s = b_1+\cdots+b_t$, then the net change vector for every $S\in \JS(\bfa,\bfb,n,m)$ is 
 \begin{equation}\label{eq:deltas}
 \delta(\bfa,\bfb,n,m)=a_1\e_1+\cdots+a_s\e_s-(b_1\e_{n+1}+\cdots + b_t\e_{n+t}).
 \end{equation}
\end{theorem}

\begin{proof} Recall that magic balls cannot be distributed.
Consistently label the $b$ balls in a juggling sequence $S\in \JS(\bfa,\bfb,n,m)$ using the labels $1,\ldots,b$ so that balls on height $j\geq 2$ at time $i-1$ are on height $j-1$ at time $i$ and the set of (necessarily non-magic) labeled balls thrown at time $i$ are distributed to the multiset of heights $J_i$. The constraint that $a_1+\cdots+a_s = b_1+\cdots+b_t$ ensures that no magic balls disappear without cancelling a non-magic ball.

A non-magic ball that is thrown at some point and that starts at height $i$ at time $0$ and ends at height $j$ at time $n$ will cumulatively generate a contribution to $\delta(S)$ of $\e_i-\e_{n+j}$ because the intermediate contributions from throwing cancel each other in a telescoping manner.  If a magic ball interacts with a non-magic ball at time $i$ and at height $j$, then this magic ball started at height $i+j$ at time $0$. If the non-magic ball started at height $h$ then the contribution to $\delta(S)$ from this pair of balls will be $\e_h-\e_{i+j}$, again because intermediate contributions cancel in a telescoping manner. 
The contribution to $\delta(S)$ of a ball that is not thrown is necessarily zero, as is its contribution to the right-hand side of Equation~\eqref{eq:deltas} because a ball that starts at height $h$ ends at height $h-n$ contributes $\e_h-\e_{n+(h-n)}$.
Summing the contributions over all balls gives Equation~\eqref{eq:deltas}.  This does not depend on the choice of $S$.
\end{proof}

If the dimension of $\delta(S)$ is $r+1$, we can view $\delta(S)$ as a weight of the Lie algebra of type $A_r$.

\begin{theorem}\label{thm:prop:cor}
Let $r\geq 1$.  Consider an initial state $\bfa=\langle a_1,\ldots,a_s\rangle$ and a terminal state $\bfb=\langle b_1\rangle$ satisfying $s\leq r+1$ and $a_1+\cdots+a_s = b_1$. When $\delta=a_1\e_1+\cdots+a_{s}\e_{s}-b_1\e_{r+1}$,
the function $\Gamma:P_{A_r}(\delta)\to \JS(\bfa,\bfb,r)$ is a bijection.
\end{theorem}
\begin{proof}
To show $\Gamma$ is a bijection, we need only show that both  $\Gamma$ and its inverse are well defined on their respective domains. We first show that $\Gamma$ is well defined.

Since $s\leq r+1$, we write $\bfa=\langle a_1,a_2,\ldots,a_{r+1}\rangle$, where we append zeros so $\bfa$ has $r+1$ entries. Let $p$ be an arbitrary partition of \[\delta=a_1\e_1+\cdots+a_{r}\e_{r}-(b_1-a_{r+1})\e_{r+1}\] into positive roots of the Lie algebra of type $A_r$, and further subpartition $p$ into multisets $p_i$ for $1\leq i\leq r$ consisting of all positive roots in $p$ of the form $\e_i-\e_j$ for some $i<j\leq r+1$. Let $m(i,j)$ denote the multiplicity of $\e_i-\e_j$ in the set $p_i$. 
We construct a juggling sequence $S=(\bfs_0,\hdots,\bfs_r)$ as follows.  Define
\begin{align} \label{map:partition2jugglingsequence}
    \bfs_{i}&=\begin{cases}\langle a_1,a_2,a_3,\ldots, a_{r+1}\rangle&\mbox{if $i=0$}\\\langle a_{i+1},a_{i+2},\ldots, a_{r+1}\rangle+\left\langle\displaystyle\sum_{j=1}^im(j,i+1),\displaystyle\sum_{j=1}^im(j,i+2),\ldots,\displaystyle\sum_{j=1}^im(j,r+1)\right\rangle&\mbox{if $1\leq i\leq r$.}
    \end{cases}
\end{align}
(In Definition \ref{def:gamma}, we append 0's to the shorter vector as necessary to make the addition well defined.) In the juggling sequence we are constructing, at time $i$ we will have $m(i,k)$ balls thrown to height $k$ and
\begin{align}
    \bfs_r&=\left\langle a_{r+1}+\displaystyle\sum_{j=1}^rm(j,r+1)\right\rangle.
\end{align}
Thus  
$\sum_{j=1}^rm(j,r+1)=b_1-a_{r+1}$, since this sum counts the number of positive roots containing $-\e_{r+1}$. Therefore $\bfs_r=\bfb$, so $\Gamma(p)\in \JS(\bfa,\bfb,r)$. 

It is easier to see that the inverse function is well defined; it follows directly from Theorem~\ref{thm:netheight}.
\end{proof}

In the next section we start with an arbitrary juggling sequence and determine the corresponding partition function.  
Before doing so we present results for calculating Kostant's partition functions of Type $A$ for various weights using juggling sequences.

\begin{theorem}
\label{thm:unresAA}
Let $\mu$ be a weight of the Lie algebra of type $A_r$ and let $\langle\mu_1,\ldots,\mu_{r+1}\rangle$ be its standard basis vector representation. The function $\Gamma$ is a bijection between $P_{A_r}(\mu)$ and \[\JS(\langle\mu_1,\ldots,\mu_{r}\rangle, \langle \mu_{1}+\cdots+\mu_r \rangle, r).\]
\end{theorem}
\begin{proof}
Apply Theorem~\ref{thm:prop:cor} where $\mu_{r+1} = -(\mu_1+\cdots+\mu_r)$ and 
the net change vector 
\[\delta(\langle\mu_1,\ldots,\mu_{r}\rangle, \langle \mu_1+\cdots+\mu_r \rangle, r)\] is \[\langle \mu_{1}, \ldots, \mu_{r},0\rangle - \langle 0,\ldots,0,\mu_1+\cdots+\mu_r\rangle 
= \langle \mu_{1},  \ldots, \mu_{r},\mu_{r+1}\rangle.\qedhere\] 
\end{proof}

\begin{theorem}
\label{thm:unresA}
Let $\mu$ be a weight of the Lie algebra of type $A_r$ and let $\langle\mu_1,\ldots,\mu_{r+1}\rangle$ be its standard basis vector representation. Then
\[\wp_{A_r}(\mu) = \js(\langle\mu_1,\ldots,\mu_{r}\rangle, \langle \mu_{1}+\cdots+\mu_r \rangle, r).\]
\end{theorem}

\begin{corollary}\label{cor:hrootAtoMJS}
Let $n\in\mathbb{N}$ and $\wp_{A_r}$ denote the type $A_r$ Kostant's partition function. If $\hroot_{A_r}$ is the highest root of the Lie algebra of type $A_r$, then 
$$\wp_{A_r}(n\hroot_{A_r}) = \js(\langle n \rangle, \langle n \rangle, r).$$
\end{corollary}

\begin{proof}
The highest root of $A_r$ is $\hroot_A = \alpha_1 + \cdots + \alpha_r$, so $n\hroot_A = n\e_1 - n\e_{r+1}$. 
\end{proof}

An alternate proof of Corollary \ref{cor:hrootAtoMJS} was proposed by Steve Butler \cite{Bprivate}; we include his beautiful combinatorial argument 
involving strips of paper. 

\begin{proof}[Alternate proof of Corollary \ref{cor:hrootAtoMJS}]
We enumerate the partitions of $n\hroot_{A_r}$ using all the positive roots of type $A_r$ in the following way. Take $n$ indistinguishable strips of paper, and write the numbers $1$ through $r$ on each of them. Each strip can be torn into pieces with tears between the numbers on the strips. Each tearing corresponds to a partition of $n\hroot_{A_r}$ in which a substrip with $i, i+1, \ldots, j$ is associated to the positive root $\alpha_i + \alpha_{i+1} + \cdots + \alpha_j$.  Any partition of $n\hroot_{A_r}$ can be uniquely associated with a multiset of torn substrips of paper made by tearing the $n$ strips of paper labeled $1$ through $r$, and vice versa.

We can also uniquely associate each element of $\JS(\langle n \rangle, \langle n \rangle, r)$ to a tearing of $n$ strips of paper labeled $1$ through $r$. Let $S = ( \bfs_0, \bfs_1, \ldots, \bfs_r )$ be any juggling sequence in $\JS(\langle n \rangle, \langle n \rangle, r)$, so that all balls are at height 1 in $\bfs_0$ and $\bfs_r$.  Associate a ball to a strip of paper. When the ball is thrown at time $i$ to height $j$ tear its associated strip of paper between numbers $i+j$ and $i+j+1$.  
Each strip of paper describes the movement of the associated ball. Thus, a multiset of torn substrips of paper made by tearing $n$ strips of paper labeled $1$ through $r$ yields a unique juggling sequence in $\JS(\langle n \rangle, \langle n \rangle, r)$. 
This proves that $\wp_{A_r}(n\hroot_{A_r}) = \js(\langle n \rangle, \langle n \rangle, r)$.
\end{proof}

\begin{example}
The partition $\{\alpha_1,\alpha_2+ \alpha_3, \alpha_4, \alpha_1 + \alpha_2, \alpha_3 + \alpha_4, \alpha_1, \alpha_2, \alpha_3, \alpha_4\}$ of $3\tilde{\alpha}_{A_4}$ corresponds to the multiset of torn strips numbered $1,2,3,4$ shown in Figure~\ref{fig:stripsToSeq}. The equivalent juggling sequence in  $\JS(\langle 3 \rangle, \langle 3 \rangle, 4)$ is also pictured.
\begin{figure}[htbp]
    \centering
\resizebox{1.75in}{!}{
    \begin{tikzpicture}
\draw (0.0,0.0) rectangle (1.0,-1.0) node[pos=.5] {1};
\draw (1.15,0.0) rectangle (2.15,-1.0) node[pos=.5] {2};
\draw (2.15,0.0) rectangle (3.15,-1.0) node[pos=.5] {3};
\draw[line width = 0.5mm,white] (2.15, 0.0) -- (2.15, -1.0);
\draw (3.3,0.0) rectangle (4.3,-1.0) node[pos=.5] {4};

\draw (0.0,-1.1) rectangle (1.0,-2.1) node[pos=.5] {1};
\draw (1.0,-1.1) rectangle (2.0,-2.1) node[pos=.5] {2};
\draw[line width = 0.5mm,white] (1.0, -1.1) -- (1.0, -2.1);
\draw (2.3,-1.1) rectangle (3.3,-2.1) node[pos=.5] {3};
\draw (3.3,-1.1) rectangle (4.3,-2.1) node[pos=.5] {4};
\draw[line width = 0.5mm,white] (3.3, -1.1) -- (3.3, -2.1);

\draw (0.0,-2.2) rectangle (1.0,-3.2) node[pos=.5] {1};
\draw (1.1,-2.2) rectangle (2.1,-3.2) node[pos=.5] {2};
\draw (2.2,-2.2) rectangle (3.2,-3.2) node[pos=.5] {3};
\draw (3.3,-2.2) rectangle (4.3,-3.2) node[pos=.5] {4};
\end{tikzpicture}}\\\vspace{2mm}
\begin{tikzpicture}
    \draw[<->][line width=0.45mm] (6.35, 2) -- (6.35, 2.75);
\end{tikzpicture}\\\vspace{2mm}
\begin{tikzpicture}
\draw (0,-1.5) -- (0,.5);
\draw (0, -1.0) -- (1.0, -1.0) -- (1.5, -0.5);
\fill (0.2,-0.8) circle (3pt);
\fill (0.5,-0.8) circle (3pt);
\fill (0.8,-0.8) circle (3pt);
\draw (0, -0.5) -- (1.0, -0.5) -- (1.5, 0.0);
\draw (0, 0.0) -- (1.0, 0.0) -- (1.5, 0.5);
\draw (0, -2.0) node {$\langle3\rangle$};
\draw[->] (1.8,0) -- (2.3,0);
\draw (2.8,-1.5) -- (2.8,.5);
\draw (2.8, -1.0) -- (3.8, -1.0) -- (4.3, -0.5);
\fill (3.0,-0.8) circle (3pt);
\fill (3.3,-0.8) circle (3pt);
\draw (2.8, -0.5) -- (3.8, -0.5) -- (4.3, 0.0);
\fill (3.0,-0.3) circle (3pt);
\draw (2.8, 0.0) -- (3.8, 0.0) -- (4.3, 0.5);
\draw (2.8, -2.0) node {$\langle2,1\rangle$};
\draw[->] (4.6,0) -- (5.1,0);
\draw (5.6,-1.5) -- (5.6,.5);
\draw (5.6, -1.0) -- (6.6, -1.0) -- (7.1, -0.5);
\fill (5.8,-0.8) circle (3pt);
\fill (6.1,-0.8) circle (3pt);
\draw (5.6, -0.5) -- (6.6, -0.5) -- (7.1, 0.0);
\fill (5.8,-0.3) circle (3pt);
\draw (5.6, 0.0) -- (6.6, 0.0) -- (7.1, 0.5);
\draw (5.6, -2.0) node {$\langle2,1\rangle$};
\draw[->] (7.4,0) -- (7.9,0);
\draw (8.4,-1.5) -- (8.4,.5);
\draw (8.4, -1.0) -- (9.4, -1.0) -- (9.9, -0.5);
\fill (8.6,-0.8) circle (3pt);
\fill (8.9,-0.8) circle (3pt);
\draw (8.4, -0.5) -- (9.4, -0.5) -- (9.9, 0.0);
\fill (8.6,-0.3) circle (3pt);
\draw (8.4, 0.0) -- (9.4, 0.0) -- (9.9, 0.5);
\draw (8.4, -2.0) node {$\langle2,1\rangle$};
\draw[->] (10.2,0) -- (10.7,0);
\draw (11.2,-1.5) -- (11.2,.5);
\draw (11.2, -1.0) -- (12.2, -1.0) -- (12.7, -0.5);
\fill (11.4,-0.8) circle (3pt);
\fill (11.7,-0.8) circle (3pt);
\fill (12.0,-0.8) circle (3pt);
\draw (11.2, -0.5) -- (12.2, -0.5) -- (12.7, 0.0);
\draw (11.2, 0.0) -- (12.2, 0.0) -- (12.7, 0.5);
\draw (11.2, -2.0) node {$\langle3\rangle$};
\end{tikzpicture}
    \caption{Illustration of the association of a set of torn strips of paper to a  juggling sequence in $\JS(\langle 3 \rangle, \langle 3 \rangle, 4)$. }
    \label{fig:stripsToSeq}
\end{figure}
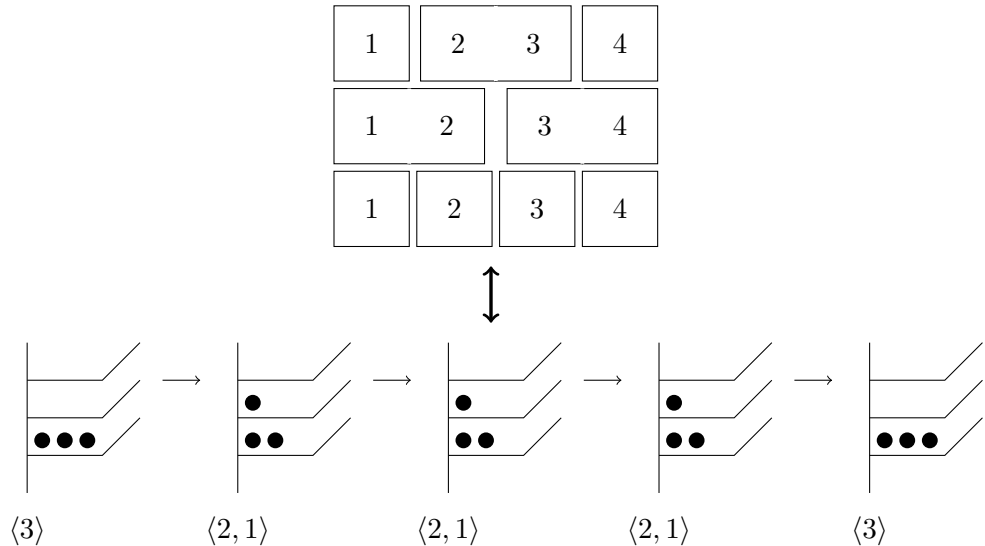
\end{example}

\section{Bijection restrictions}\label{sec:restrict}
We gain further insights about both Kostant's partition functions and juggling sequences by understanding restrictions to $\Gamma$ from Section~\ref{sec:bijection}.  
A key insight is that a restriction on the set of roots that can appear in a partition corresponds to a restriction on which throws are allowed in a juggling sequence.  Theorem~\ref{thm:restrictparts} explains which juggling sequences correspond to a restriction of the positive roots for Kostant's partition function while Theorem~\ref{thm:jugtopart} gives the restricted partition function that corresponds to enumerating general juggling sequences.

\begin{definition}
Let $\T\subseteq \T_r$ be a set of allowed throws.  Define $\JS_\T(\bfa,\bfb,n,m)$ to be the set of juggling sequences that only use allowed throws and $\js_\T(\bfa,\bfb,n,m)=\lvert\JS_\T(\bfa,\bfb,n,m)\rvert$.
\end{definition}

\begin{theorem}\label{thm:restrictparts}
Let $\mu$ be a weight of the Lie algebra of type $A_r$ and let $\Lambda\subseteq \Phi_{A_r}^+$. Then $P_\Lambda(\mu)$ is in bijection with $\JS_{\Gamma(\Lambda)}(\langle \mu_{1}, \mu_{2}, \ldots, \mu_{r} \rangle, \langle  \mu_{1}+\mu_{2}+ \cdots+ \mu_{r}  \rangle, r)$ and
\[K_\Lambda(\mu)=\js_{\Gamma(\Lambda)}(\langle \mu_{1}, \mu_{2}, \ldots, \mu_{r} \rangle, \langle \mu_{1}+\mu_{2}+ \cdots+ \mu_{r} \rangle, r).\]
\end{theorem}

Note that Theorem~\ref{thm:prop:cor} is a special case of Theorem~\ref{thm:restrictparts} when $\Lambda=\Phi_{A_r}^+$.

\begin{proof}
The restriction that a root $\e_i-\e_{i+j}$ not be allowed in the partition is equivalent to the condition that $m(i,i+j)$ must be zero in the proof of Theorem~\ref{thm:prop:cor}, which is the same as requiring that no balls are thrown at time $i$ to level $j$.
\end{proof}

Theorem~\ref{thm:restrictparts} answers a question of Butler and Graham \cite[Section 4.2]{BG}, in which they wish to incorporate a restriction of the maximum height to which any ball is thrown.  If we restrict our throws to height $h$, this corresponds to restricting the positive roots to be of the form $\e_i-\e_{i+j}$ where $j\leq h$.  

We present an example where roots are restricted to be either simple roots or the sum of three consecutive simple roots.

\begin{example}
Let $\Lambda = \{\alpha_1, \alpha_2, \alpha_3, \alpha_4, \alpha_1 + \alpha_2 + \alpha_3, \alpha_2 + \alpha_3 + \alpha_4\} \subseteq \Phi_{A_4}^+$. Then $\T=\Gamma(\Lambda)$ is the set of throws at any time to heights $1$ or $3$.

Let $\mu = \alpha_1 + 2\alpha_2 + 2\alpha_3 + \alpha_4$.  Then $K_\Lambda(\mu)=4$ because the partitions of $\mu$ using the positive roots in $\Lambda$ are 
\begin{gather*}
 \{\alpha_1,\,\alpha_2,\,\alpha_2,\,\alpha_3,\,\alpha_3,\,\alpha_4\},\,
\{\alpha_1 + \alpha_2 + \alpha_3,\,\alpha_2 + \alpha_3 + \alpha_4\}, \\
   \{\alpha_1 + \alpha_2 + \alpha_3,\,\alpha_2,\,\alpha_3,\,\alpha_4\}, \textup{ and }
    \{\alpha_1,\,\alpha_2,\,\alpha_3,\,\alpha_2 + \alpha_3 + \alpha_4\}.
\end{gather*}

Since 
\[\mu = \alpha_1 + 2\alpha_2 + 2\alpha_3 + \alpha_4 = \langle 1, 1,0,-1,-1 \rangle,\]
Theorem~\ref{thm:restrictparts} implies that $\js_{\T}(\langle 1, 1, 0, -1 \rangle, \langle 1 \rangle, 4)=4$.  The four juggling sequences are
\begin{gather*}
    (\langle 1, 1, 0, -1 \rangle, \langle 2, 0, -1 \rangle, \langle 2, -1 \rangle, \langle 1 \rangle, \langle 1 \rangle), (\langle 1, 1, 0, -1 \rangle, \langle 2, 0, -1 \rangle, \langle 1, -1, 1 \rangle, \langle 0, 1 \rangle, \langle 1 \rangle), \\
    (\langle 1, 1, 0, -1 \rangle, \langle 1 \rangle, \langle 0,0,1 \rangle, \langle 0, 1 \rangle, \langle 1 \rangle), \textup{ and }
    (\langle 1, 1, 0, -1 \rangle, \langle 1 \rangle, \langle 1 \rangle, \langle 1 \rangle, \langle 1 \rangle).
\end{gather*}
The reader can verify that the only throws are of height one and three.
\end{example}

We now consider a general juggling sequence question and find the corresponding partition function question. 

\begin{theorem}\label{thm:jugtopart}
Consider juggling states $\bfa=\langle a_1,\ldots,a_s\rangle$ and $\bfb=\langle b_1,\ldots,b_t\rangle$, satisfying \[a_1+\cdots+a_s = b_1+\cdots+b_t,\] and positive integers $n$ and $m$. Let
\[
 \delta=a_1\e_1+\cdots+a_s\e_s-(b_1\e_{n+1}+\cdots + b_t\e_{n+t})
\]
and $r+1$ be the dimension of this vector. Further, let
\[
\Lambda = \big\{\e_i-\e_j\in\Phi_{A_r}^+: 1\leq i < j\leq r+1 \mbox{ and } i\leq n\big\},
\]
and define
\begin{align}
    Q_\Lambda(\delta) := \{p\in P_\Lambda(\delta): \text{for all $j > 0$, } a_j + \sum_{\e_i-\e_j\in p} 1 \leq m\}. \label{PT'}
\end{align}
Then $\Gamma$ is a bijection between the set $Q_\Lambda(\delta)$ and $\JS(\bfa,\bfb,n,m)$ and therefore $\js(\bfa,\bfb,n,m) = |Q_\Lambda(\delta)|$. When there is no hand capacity constraint, then $Q_\Lambda(\delta)=P_\Lambda(\delta)$ and $\js(\bfa,\bfb,n) = K_\Lambda(\delta)$.
\end{theorem}

\begin{proof}
Similar to the proof of Theorem~\ref{thm:prop:cor}, we show that both $\Gamma$ and its inverse are well defined on their domains.

We first show $\Gamma$ is well defined. Let $p\in Q_\Lambda(\delta)$ with parts $\e_i-\e_j$ satisfying
\begin{enumerate}
    \item $1\leq i<j\leq r+1$ and $i\leq n$
    \item for all $j>0$, $a_j+\sum_{\e_i-\e_j\in p}1\leq m$.
\end{enumerate}
Now, similar to construction of $\bfs_i$ in Equation~\eqref{map:partition2jugglingsequence}, create a juggling sequence $S$ starting at $\bfa$ of length $n$, whose throws are at time $i$ to height $j-i$ for all $\e_i-\e_j\in p$. The terminal state of $S$ is given by $\bfs=\langle s_n,s_{n+1},\ldots,s_{n+t-1}\rangle$ where $s_i=a_i+\sum_{\e_i-\e_j\in p} 1=b_i$ since $p$ is a partition of $\delta$. The condition in item (2) restricts the number of balls present at any time and at any height to be less than or equal to $m$.  Therefore $S\in \JS(\bfa,\bfb,n,m)$.

We now show that $\Gamma^{-1}$ is well defined. By construction, $\delta=\delta(\bfa,\bfb,n,m)$ is the net change vector of every juggling sequence in $\JS(\bfa,\bfb,n,m)$.  Every juggling sequence $S$ in  $\JS(\bfa,\bfb,n,m)$ corresponds to a multiset of throws, where each of the throws is at a time $i$ between $1$ and $n$ to a height of at most $r+1-i$ and where at most $m$ non-magic balls are at any height at any time.  The condition on throws implies that $\Gamma^{-1}$ takes $\JS(\bfa,\bfb,n,m)$ into $P_\Lambda(\delta)$. The hand capacity condition shows that $\Gamma^{-1}(\JS(\bfa,\bfb,n,m))\subseteq Q_{\Lambda}(\delta)$ because the number of balls at height $h$ at time $j-h$ is the number of balls that started at height $j$ (which is negative if they are magic) plus the number of balls that were thrown to height $k$ at time $j-k$ for $1\leq k<j-h$.  Under $\Gamma^{-1}$, these throws in $S$ are of the form $(\e_i-\e_j)$ for $i<j$.  The hand capacity restriction on $S$ is exactly the restriction on the number of times a positive root can be used, which defines $Q_\Lambda(\delta)$. 

When there is no hand capacity restriction,  $Q_{\Lambda}(\delta)=P_{\Lambda}(\delta)$, so $\Gamma$ becomes a bijection between $P_{\Lambda}(\delta)$ and $\JS(\bfa,\bfb,n)$ and we conclude that $\js(\bfa,\bfb,n)=K_\Lambda(\delta)$.
\end{proof}

Butler and Graham give generating functions for the number of some periodic  juggling sequences, which are juggling sequences whose terminal state is the same as its initial state. With Theorem \ref{thm:jugtopart} at hand, we now present the weights of a Lie algebra of type $A_r$ to which these juggling sequences correspond. This is presented in Table \ref{tab:BG}, with an additional column listing the corresponding weights, keeping in mind the restriction to the partition function we are using, as described in Equation \eqref{PT'}. 
By extracting coefficients from the generating functions we are able to give some closed partition function formulas.

\begin{table}[htbp]
    \centering
    \resizebox{\textwidth}{!}{
    \begin{tabular}{cccr}
        \hline
         State  &   $m$  &    Generating Function &   Weight\\
         \hline
         $\langle 2 \rangle$  &   $2$ &$\displaystyle\frac{x - 2x^2}{1-5x+5x^2}$    &
         $2\alpha_1 + \cdots + 2\alpha_n$
         \\

         $\langle1,1\rangle$  &   $2$ &$\displaystyle\frac{x-2x^2+x^3}{1-5x+5x^2}$  &
         $\begin{cases}
         \alpha_1 + \alpha_2  &\mbox{ if $n=1$}\\
         \alpha_1 + 2\alpha_2 + \cdots + 2\alpha_n + \alpha_{n+1} & \mbox{if $n\geq 2$}
         \end{cases}$\\

         $\langle2,1\rangle$  &   $2$ &$\displaystyle\frac{x-4x^2+3x^3}{1-8x+13x^2}$&
         
         $\begin{cases}
         2\alpha_1 + \alpha_2 & \mbox{if $n=1$}\\
         2\alpha_1+3\alpha_2+\cdots+3\alpha_n+\alpha_{n+1} & \mbox{if $n\geq 2$}
         \end{cases}$\\

         $\langle1,1,1\rangle$&   $2$ &$\displaystyle\frac{x-5x^2+7x^3}{1-8x+13x^2}$&
         $\begin{cases}
         \alpha_1 + \alpha_2 + \alpha_3   &   \text{if $n=1$}\\
         \alpha_1 + 2\alpha_2 + 2\alpha_3 + \alpha_4  & \text{if $n=2$}\\
         \alpha_1 + 2\alpha_2 + 3\alpha_3 + \cdots + 3\alpha_{n} + 2\alpha_{n+1} + \alpha_{n+2} & \text{if $n\geq 3$}
         \end{cases}$\\

         $\langle2,2\rangle$  &   $2$ &$\displaystyle\frac{x-11x^2+33x^3-27x^4}{1-14x+54x^2-57x^3}$&
         $\begin{cases}
         2\alpha_1 + 2\alpha_2 & \text{if $n=1$}\\
         2\alpha_1 + 4\alpha_2 + \cdots + 4\alpha_n + 2\alpha_{n+1} & \text{if $n\geq 2$}
         \end{cases}$\\

         $\langle3\rangle$    &   $3$ &$\displaystyle\frac{x-6x^2+7x^3}{1-10x+27x^2-20x^3}$&
         $3\alpha_1 + \cdots + 3\alpha_n$\\

         $\langle2,1\rangle$  &   $3$ &$\displaystyle\frac{x-5x^2+7x^3-3x^4}{1-10x+27x^2-20x^3}$&
         $\begin{cases}
         \alpha_1 + \alpha_2 & \text{if $n=1$}\\
         \alpha_1 + 2\alpha_2 + \cdots + 2\alpha_n + \alpha_{n+1} & \text{otherwise}
         \end{cases}$\\
         
         $\langle4 \rangle$    &   $4$ &$\displaystyle\frac{x-15x^2 + 70x^3-127x^4+78x^5}{1-20 x + 135 x^2 -396 x^3 + 518 x^4 - 245 x^5}$&
         $4\alpha_1 + \cdots + 4\alpha_n$\\[10pt]
         
         $\langle5 \rangle$    &   $5$ &$\displaystyle\frac{x-30x^2+320x^3-1604x^4+4059x^5-4970x^6+2320x^7}{1-36 x + 480 x^2 -3140 x^3 + 11059 x^4 - 21180 x^5 + 20560 x^6 -7840 x^7}\hspace{-1.75in}$&
         $5\alpha_1 + \cdots + 5\alpha_n$\\[10pt]
         \hline
    \end{tabular}
    }
    \caption{Number of periodic sequences of length $n$ for a set of specific initial and terminal states, with their corresponding weights. The first three columns in this table first appeared in~\cite{BG}. The denominators of the generating functions for the states $\langle n \rangle$  appeared in \cite{CRY}, the numerators for $n=4,5$ were then obtained from initial terms computed with the code in \cite{BBCV}.}
    \label{tab:BG}
\end{table}

In what follows, for $\Lambda\subseteq\Phi_{A_r}^+$, we let 
\[
    Q_{\Lambda,m}(\mu) = \{p\in P_{\Lambda}(\mu): \text{for all $j > 0$, } a_j + \sum_{\e_i-\e_j\in p} 1 \leq m\}. 
\]

\begin{corollary}
If $r\geq 2$, then
\[K_{A_r}(2\a_1+2\a_2+\cdots+2\a_r)=
\frac{1}{5\cdot 2^{r}} 
\big((5-\sqrt{5})^r+(5+\sqrt{5})^r\big).\]
\end{corollary}

\begin{remark} \label{rem: cry recurrence}
For a positive integer $n$, in \cite[Thm. 1]{CRY} it was shown that the sequence $\{K_{A_r}( n\hroot_{A_r})\}_{r\geq 1}$ satisfies a linear recurrence of degree equal to the number of integer partitions of $n$. It would be of interest to find the generating function for this sequence for any fixed $r$. (See Table~\ref{tab:BG} and Section~\ref{sec:genfunctions}.)
\end{remark}

\begin{corollary}
Let $r\geq 2$. If  $\Lambda=\Phi_{A_r}^+\setminus\{\a_r\}$, then 
\[\big\lvert Q_{\Lambda,2}(\a_1+2\a_2+\cdots+2\a_{r-1}+\a_r)\big\rvert=\frac{\left(\sqrt{5}-3\right) \left(5-\sqrt{5}\right)^{r-1}+\left(\sqrt{5}+3\right) \left(5+\sqrt{5}\right)^{r-1}}{5 \sqrt{5}\cdot 2^{r} }.
\]
\end{corollary}

\begin{corollary}
Let $r\geq 3$. If $\Lambda=\{\e_i-\e_j\in \Phi_{A_r}^+: 1\leq i<j\leq r+1, i\leq r-2\}$, then 
\[\big\lvert Q_{\Lambda,2}(2\a_1+3\a_2+\cdots+3\a_{r-1}+\a_r)\big\rvert=\frac{\left(14 \sqrt{3}-9\right) \left(4-\sqrt{3}\right)^{r-1}+\left(14 \sqrt{3}+9\right) \left(4+\sqrt{3}\right)^{r-1}}{169 \sqrt{3}}.
\]
\end{corollary}

\begin{corollary}
Let $r\geq 5$. If $\Lambda=\{\e_i-\e_j\in \Phi_{A_r}^+: 1\leq i<j\leq r+1, i\leq r-2\}$, then 
{\small\[\big\lvert Q_{\Lambda,2}(\a_1+2\a_2+3\a_3+\cdots+3\a_{r-2}+2\a_{r-1}+\a_r)\big\rvert=\frac{\left(9-14 \sqrt{3}\right) \left(4-\sqrt{3}\right)^{r-2}+\left(9+14 \sqrt{3}\right) \left(4+\sqrt{3}\right)^{r-2}}{338}.
\]}
\end{corollary}

Next we place restriction to the throws of juggling sequences. We show that such restrictions correspond to restrictions on the positive roots allowed in Kostant's partition function. We formalize this in the next result, which follows similarly to Theorem~\ref{thm:jugtopart} using the correspondence between allowed throws and allowed roots.

\begin{corollary}
\label{cor:resMJS}
Consider juggling states $\bfa=\langle a_1,\ldots,a_s\rangle$ and $\bfb=\langle b_1,\ldots,b_t\rangle$, satisfying \[a_1+\cdots+a_s = b_1+\cdots+b_t,\] and positive integers $n$ and $m$. Let
\[
 \delta=a_1\e_1+\cdots+a_s\e_s-(b_1\e_{n+1}+\cdots + b_t\e_{n+t})
\]
and $r+1$ be the dimension of this vector. Suppose $\T\subseteq\T_r$ is an allowed set of throws and 
\[
\Lambda = \big\{\Gamma^{-1}(T):T\in\T \big\}\subseteq\Phi_{A_r}^+,
\]
and define
\[
    Q_\Lambda(\delta) := \{p\in P_\Lambda(\delta): \text{for all $j > 0$, } a_j + \sum_{\e_i-\e_j\in p} 1 \leq m\}.
\]
The set $\JS_\T(\bfa,\bfb,n,m)$ is in bijection with $Q_\Lambda(\delta)$; therefore, $\js_T(\bfa,\bfb,n,m) = |Q_\Lambda(\delta)|.$  
\end{corollary}

Similar to above, $\js(\bfa,\bfb,n)=K_{\Lambda}(\delta)$.

\section{Other Lie types}\label{sec:HIO}
In this section we extend our work on  juggling sequences to the Lie algebras of type $B$, $C$, and $D$, by defining how juggle in these Lie types. We then give a correspondence of counting these new juggling sequences as sums of juggling sequences in type $A$. We end the section by giving a result relating the number of partitions of the highest root, settling a problem of Harris, Insko, and Omar~\cite{HIO}.

\subsection{Juggling sequences of other Lie types}\label{sec:othertypes-juggling}

We can define juggling sequences of types $B$, $C$, and $D$ that involve a second (reflected) conveyer.  To be precise, a type $B$, $C$, or $D$ juggling state is a pair $(\bfs,\bft)$ where $\bfs = \langle s_1, \ldots, s_h \rangle$ is an integer vector and $\bft = \langle t_1, \ldots, t_h \rangle$ is a nonnegative integer vector. The difference between types $B$, $C$, and $D$ is how balls can be thrown, be dropped, or disappear, which are informed by the positive roots of each Lie type.  

\begin{definition}\label{defs:juggling other types}
A {\em downward throw} at time $i$ to height $j$ is a throw of a ball at height $1$ in the standard conveyer to height $j$ in the reflected conveyer.  This corresponds to a positive root of the form $\e_i+\e_{i+j}$.  A {\em single drop} at time $i$ is when a ball at height $1$ at time $i$ in the standard conveyer disappears.  This corresponds to a positive root of the form $\e_i$.  A {\em double drop} at time $i$ is when two balls at height $1$ at time $i$ in the standard conveyer disappear simultaneously.  This corresponds to a positive root of the form $2\e_i$. Last, a {\em cancellation} at time $i$ is when a ball at height $1$ in the standard conveyer and a ball at height $1$ in the reflected conveyer disappear simultaneously. 
\end{definition}

Figure~\ref{fig:BCDthrows} illustrates the behaviors described in Definition \ref{defs:juggling other types}.  

\begin{figure}[ht]
    \centering
    \subcaptionbox{Downward throw to height $3$}%
  [.23\textwidth]{\includegraphics[height=1in]{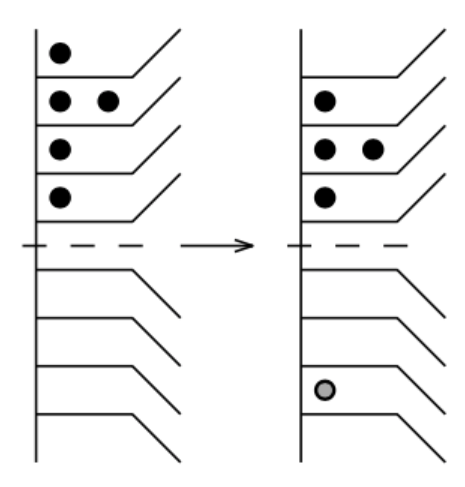}}
    \subcaptionbox{Single drop}%
  [.23\textwidth]{\includegraphics[height=1in]{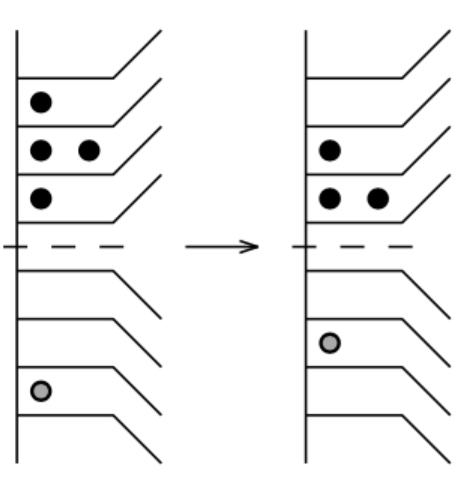}}
    \subcaptionbox{Double drop}%
  [.23\textwidth]{\includegraphics[height=1in]{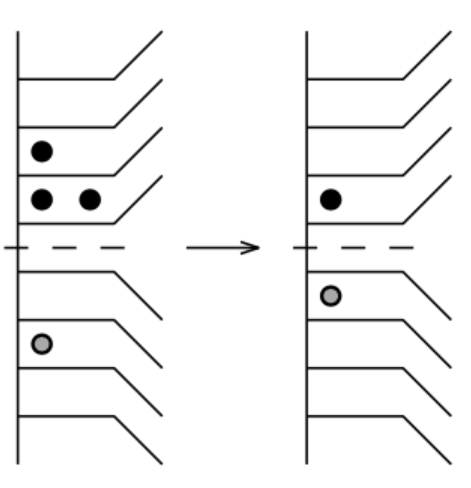}}
    \subcaptionbox{Cancellation}%
  [.23\textwidth]{\includegraphics[height=1in]{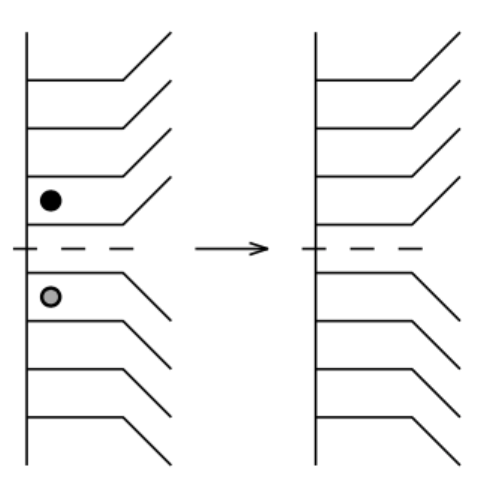}}
    \caption{Examples of behaviors that may occur in juggling sequences of types $B$, $C$, and $D$.}
    \label{fig:BCDthrows}
\end{figure}

\begin{remark}
Balls in the reflected conveyer are never thrown; they must disappear through cancellations.  
\end{remark}

\begin{definition}
A {\em type $B$, $C$, or $D$ juggling sequence} is a sequence \[S=\big((\bfs_0,\bft_0), (\bfs_1,\bft_1), \ldots, (\bfs_n,\bft_n) \big)\] where two successive juggling states $(\bfs_{i-1},\bft_{i-1})=(\langle s_1, \ldots, s_h \rangle,\langle t_1, \ldots, t_h \rangle)$ and $(\bfs_{i},\bft_{i})$ satisfy 

\begin{equation*}
    \bfs_i=\langle s_2 + b_1, s_3 + b_2, \ldots, s_h + b_{h-1}, b_h, \ldots, b_{h'} \rangle
\end{equation*}
and
\begin{equation*}
    \bft_i=\langle t_2 + c_1, t_3 + c_2, \ldots, t_h + c_{h-1}, c_h, \ldots, c_{h'} \rangle,
\end{equation*}
subject to conditions that depend on the type as detailed below.

In type $D$, positive roots are of the form $\e_i-\e_j$ and $\e_i+\e_j$, which means (upward) throws and downward throws are allowed (as are cancellations) but no drops are allowed. As such, the nonnegative integers $b_j$ and $c_j$ satisfy 
\begin{equation}\label{eq:typeD}
    \sum_{j=1}^{h'} b_j+c_j = s_1-t_1.  
\end{equation}

In type $B$, positive roots are of the form $\e_i-\e_j$, $\e_i+\e_j$, and $\e_i$, so (upward) throws, downward throws, and single drops are allowed (as are cancellations), but no double drops are allowed.  The condition in Equation~\eqref{eq:typeD} is replaced by
\begin{equation}\label{eq:typeBC}
   d_i + \sum_{j=1}^{h'} b_j+c_j = s_1-t_1, 
\end{equation}
where $d_i$ is the nonnegative integer of single drops at time $i$. 

In type $C$, positive roots are of the form $\e_i-\e_j$, $\e_i+\e_j$, and $2\e_i$, so (upward) throws, downward throws, and double drops are allowed (as are cancellations), but no single drops are allowed.  The condition is that the $b_j$ and $c_j$ satisfy Equation~\eqref{eq:typeBC} with the additional restriction that $d_i$ is a nonnegative even integer.
\end{definition}

\begin{example}
In type $D_4$, $\wp_{D_4}(\tilde{\alpha}_{D_4})=5$ since the five partitions of $\tilde{\alpha}_{D_4}=2\e_1$ (and their corresponding juggling sequences) are the following.

\smallskip
\begin{center}
\resizebox{\textwidth}{!}{
    \begin{tikzpicture}
    \node at (0,0) {\includegraphics[width=6.5in]{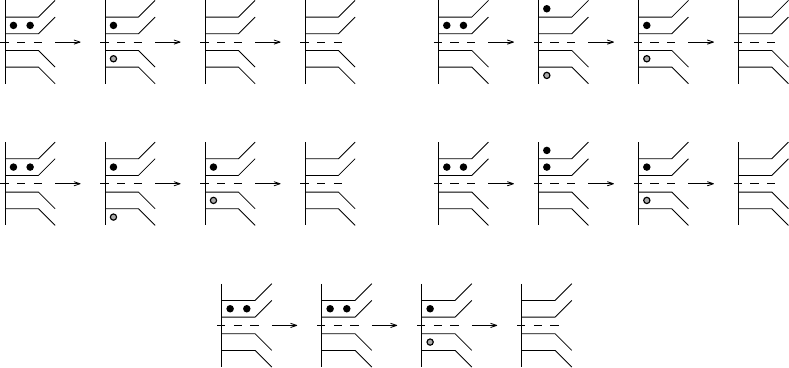}
    };
    \node at (-4.5,1.5) {$p_1=\{\e_1-\e_2,\,\e_1+\e_2\}$};
    \node at (4.5,1.5) {$p_2=\{\e_1-\e_3,\,\e_1+\e_3\}$};
    \node at (-4.5,-1.5) {$p_3=\{\e_1-\e_2,\,\e_1+\e_3,\,\e_2-\e_3\}$};
    \node at (4.5,-1.5) {$p_4=\{\e_1-\e_2,\,\e_1-\e_3,\,\e_2+\e_3\}$};
    \node at (0,-4.5) {$p_5=\{\e_1-\e_2,\,\e_1-\e_2,\,\e_2-\e_3,\,\e_2+\e_3\}$};
    \end{tikzpicture}
    }
    \end{center}
\end{example}

\subsection{Identity of Schmidt and Bincer}\label{sec:sb}
By applying an identity of Schmidt and Bincer \cite{SB} we can convert any Kostant's partition function of type $B$, $C$, or $D$ into a sum of Kostant's partition functions of type $A$. We follow \cite[Sections~2 and 4]{SB} and use Remark \ref{rem:roots-different-types}, where we think of positive roots of type $A$, as a subset of the positive roots of other Lie types. Let $S=\Phi^+_{\mathfrak{g}}$ be the set of positive roots of $\mathfrak{g}$, and consider the set of positive roots of type $A$,  $\Phi_{A_r}^+\subseteq \Phi^+_{\mathfrak{g}}$. Next we let $T=\Phi_{A_r}^+$. 

To evaluate $K_{S}(\mu)$ we note that every partition of $\mu$ is a multiset $p$ containing some parts $\beta_i$ in $T$ and some parts $\gamma_i$ in $S\setminus T$. In other words, $\mu=\mu_T+\mu_{S\setminus T}$, where $\mu_T$ has only parts in $T$, and $\mu_{S\setminus T}$ has only parts in $S\setminus T$. By defining an ordering of the positive roots in $T=\{\beta_i: 1\leq i \leq k\}$ we note that $\mu_T=\sum_{i=1}^kc_i\beta_i=\textbf{c}\cdot \boldsymbol{\beta}$ with nonnegative integers $c_i$ for all $1\leq i\leq k$. It follows that the number of ways to write $\mu$ as $\mu_T+\mu_{S\setminus T}$ for a particular fixed configuration of the coefficients $c_1,\ldots,c_k$ is given by 
\[K_{T}(\mu_{T})=K_T(\mu-\bf{c}\cdot\boldsymbol{\beta}).\]
Thus 
\[K_{S}(\mu)=\sum_{\bf{c}} K_{T}(\mu-\bf{c}\cdot\boldsymbol{\beta}),\]
where the $\bf{c}$ runs over all possible coefficient configurations such that $\mu-\bf{c}\cdot\boldsymbol{\beta}$ is a nonnegative integral combination of the positive roots in $T$.
This establishes the following result.

\begin{proposition}[\cite{SB}, Equation~(4.1)]\label{prop:SBrecurrence} If $\g$ is a classical Lie algebra of type $B_r$, $C_r$, or $D_r$, then
$$ \wp_{\g}(\mu) = \sum_{(c_1,\ldots,c_n)} \wp_{A_r}\left(\mu - \sum_{i=1}^n c_i\beta_i \right),$$
where $\beta_1,\beta_2,\ldots,\beta_n$ denotes an ordering of the positive roots of type $A_r$, and where $(c_1,\ldots,c_n)$ denotes all possible tuples of nonnegative integers such that $\mu-\sum_{i=1}^n c_i\beta_i$ is a nonnegative integral combination of positive roots in $\Phi^+_{A_r}$. 
\end{proposition}

The next result follows from Proposition~\ref{prop:SBrecurrence} and Theorem~\ref{thm:unresA}.

\begin{corollary}\label{cor:applyingSB}
Let $\g$ be a classical Lie algebra of rank $r$ and let $\beta_1,\beta_2,\ldots,\beta_n$ denote an ordering of the positive roots of type $A_r$.
For a fixed weight $\mu$ of $\g$ let $\mathcal{C}_\mu$ be the set of all coefficient configurations $(c_1,c_2,\ldots,c_n)$ such that $\mu - \sum_{i=1}^n c_i\beta_i=\mu-\textbf{c}\cdot\boldsymbol{\beta}$  
is a nonnegative integral linear combination of positive roots in $\Phi^+_{A_r}$. If $\sum_{i=1}^{r+1}d_i\varepsilon_i$, with $d_1,\ldots,d_{r+1}\in\mathbb{Z}$, is the 
standard basis representation of
$\mu-\textbf{c}\cdot\boldsymbol{\beta}$, then
$$ \wp_{\g}(\mu) = \sum_{(c_1,\ldots,c_n)\in\mathcal{C}_\mu} \js\left(\left\langle d_1,d_2,\ldots,d_r\right\rangle, \left\langle \sum_{i=1}^rd_i \right\rangle, r, \sum_{i=1}^r|d_i|\right).$$
\end{corollary}

\subsection{Problem of Harris-Insko-Omar}\label{sec:HIOsubsec}
In \cite{HIO} Harris, Insko, and Omar determined the generating functions for the number of partitions of the highest root of a Lie algebra (of types $B$, $C$, and $D$ respectively) into a sum of positive roots. With these generating function at hand they posed the following problems.

\begin{problem}\label{prob:1B}
Find a combinatorial proof of the identity from \cite[Table 1]{HIO} 
\begin{equation} \label{eq:typeBjs11}
\wp_{B_r}(\hroot_{B_r}) = \js(\langle 1, 1 \rangle, \langle 1, 1 \rangle,  r) \textup{ when } r\geq 2.
\end{equation}
\end{problem}

\begin{problem}\label{prob:1C}
Find a combinatorial proof of the identity from \cite[Table 1]{HIO}
\begin{equation} \label{eq:typeCjs2}
\wp_{C_r}(\hroot_{C_r}) = \js(\langle 2 \rangle, \langle 2 \rangle,  r)
\textup{ when } r\geq 3.
\end{equation}
\end{problem}

The original statements in \cite{HIO} include the hand capacity constraint $m=2$.  Since only two balls are being juggled, this is not a restriction on the juggling sequence so we omit it. We settle these two problems, as well as the type $D$ version, with the following results.

\begin{theorem} \label{thm:solution-problem 1}
For $r\geq 2$, there is a bijection between the set of partitions of the highest root of the Lie algebra of type $B_r$ into positive roots and the set $\JS(\langle 1,1\rangle,\langle 1,1\rangle,r)$.
\end{theorem}

\begin{theorem} \label{thm:solution-problem 2}
For $r\geq 3$, there is a bijection between the set of partitions of the highest root of the Lie algebra of  type $C_r$ into positive roots and the set $\JS(\langle 2\rangle,\langle 2\rangle,r)$.
\end{theorem}

\begin{theorem} \label{thm:solution-new problem}
For $r\geq 4$, there is a bijection between the set of partitions of the highest root of the Lie algebra of  type $D_r$ into positive roots and the set $\mathbb{Z}_5\times \JS(\langle 1\rangle,\langle 1\rangle,r-1)$.
\end{theorem}

\begin{remark}
The theorems above show that $K_{B_r}(\tilde{\alpha}_{B_r})$ and $K_{C_r}(\tilde{\alpha}_{C_r})$ are given by the number of certain juggling sequences. In the type $D$ case, a corollary of a result of \cite[Theorem 5.2]{HIO} obtained via generating functions implies that $K_{D_r}(\tilde{\alpha}_{D_r}) = 5\cdot K_{B_{r-2}}(\tilde{\alpha}_{B_{r-2}})$, which proves Theorem \ref{thm:solution-new problem}.
However, it is natural to ask if there is a more direct combinatorial proof of the formula for $K_{D_r}(\tilde{\alpha}_{D_r})$, perhaps involving a new set of juggling sequences.
\end{remark}

\emph{Proof of Theorem~\ref{thm:solution-problem 1}:} Recall that the highest root of the Lie algebras of type $A_{r+1}$ and $B_r$ are $\hroot_{A_{r+1}}=\a_1+\cdots+\a_{r+1}$ and $\hroot_{B_r}=\a_1+2\a_2+2\a_3+\cdots+2\a_r$, respectively.
We begin by establishing a bijection from the set of partitions of $\hroot_{B_{r}}$ using the positive roots in $\Phi_{B_r}^+$ to the set of partitions of  $2\hroot_{A_{r+1}} - \alpha_1 - \alpha_{r+1}$ 
using the positive roots in $\Lambda=\Phi_{A_{r+1}}^+\setminus\{\a_{r+1}\}$.
We then give a bijection between the set of partitions of $2\hroot_{A_{r+1}} - \alpha_1 - \alpha_{r+1}$ and the desired  juggling sequences.

\begin{lemma}\label{lem:BtoA}
Consider the Lie algebras of types $A_{r+1}$ and $B_r$ for $r\geq 2$, with positive roots $\Phi_{A_{r+1}}^+$ and $\Phi_{B_r}^+$, respectively. If $\Lambda=\Phi_{A_{r+1}}^+\setminus\{\a_{r+1}\}$,
then
\[\wp_{B_r}(\hroot_{B_r}) = \wp_{\Lambda}(2\hroot_{A_{r+1}} - \alpha_1 - \alpha_{r+1}).\]
\end{lemma}
\begin{proof}
Let \[\mu = 2\hroot_{A_{r+1}} - \alpha_1 - \alpha_{r+1}=\alpha_1 + 2\alpha_2 + \cdots + 2\alpha_r + \alpha_{r+1}\] and let 
$p\in P_{\Lambda}(\mu)$. Then $p$ must contain both
\begin{enumerate}
    \item a positive root $\alpha_i + \cdots + \alpha_{r+1}$, with $i<r+1$ and 
    \label{step:1}
    \item a positive root $\alpha_j + \cdots + \alpha_r$, with $j\leq r$.\label{step:2}
\end{enumerate}
 Moreover, note that $\a_i+\cdots+\a_{j-1}+2\a_j+\cdots+2\a_r$ cannot appear in $P_{\Lambda}(\mu)$. 
Define a map $g:P_{\Lambda}(\mu)\to P_{B_r}(\hroot_{B_r})$ as follows. If the parts of $p\in P_{\Lambda}(\mu)$ described in \eqref{step:1} and \eqref{step:2} satisfy $i < j$, then let $g(p)$ be the partition of $\hroot_{B_r}$ that replaces the parts $\alpha_i + \cdots + \alpha_{r+1}$ and $\alpha_j + \cdots + \alpha_r$ in $p$ by the part  $\alpha_i + \cdots + \alpha_{j-1} + 2\alpha_j + \cdots + 2\alpha_r$.  On the other hand, if $i \geq j$, then let $g(p)$ be the partition of $\hroot_{B_r}$ that replaces the part $\alpha_i + \cdots + \alpha_{r+1}$ by $\alpha_i+\cdots+\alpha_r$.  In both cases, all other positive roots in $\Phi_{A_r}^+$ are taken to their analogous positive roots in $\Phi_{B_r}^+$ by taking every simple roots $\a_i\in\Phi_{A_r}^+$ to the simple root $\a_i\in\Phi_{B_r}^+$ for all $1\leq i\leq r$, and extending this linearly to all positive roots in $\Phi_{A_r}^+$ (see Remark~\ref{rem:roots-different-types}).

For any $p\in P_{\Lambda}(\mu)$, $g(p)$ is a partition of $\hroot_{B_r}$ and the function is injective by construction. We now establish that $g$ is surjective. To see this, take a partition $p=\{\beta_1,\beta_2,\ldots,\beta_l\}$ of $\hroot_{B_r}$, where $\beta_i\in\Phi_{B_r}^+$ for all $1\leq i\leq l$. Note that either 
\begin{enumerate}
    \item[Case 1:] there exists $\beta\in p$ such that $\beta=\alpha_i + \cdots + \alpha_{j-1} + 2\alpha_j + \cdots + 2\alpha_r$ with $i<j$, or\label{case:1} 
    \item[Case 2:] there exist $\beta,\beta'\in p$ such that $\beta=\alpha_i + \cdots +\alpha_r$ and $\beta'=\alpha_j + \cdots + \alpha_r$, where we assume without loss of generality that $i\leq j$, and when $i=j$ then both of these positive roots are the simple root $\a_{r}$.
\end{enumerate}

In Case 1, consider the partition \[p'=\left(p\setminus\{\alpha_i + \cdots + \alpha_{j-1} + 2\alpha_j + \cdots + 2\alpha_r\}\right)\cup\{\a_i+\cdots+\a_{r+1},\a_j+\cdots+\a_{r}\}.\] Note $p'$ is a partition of $\mu$, hence $p'\in P_{\Lambda}(\mu)$, and $g(p')=p$ as desired. 
In Case 2, consider the partition \[p'=\left(p\setminus\{\alpha_i + \cdots + \alpha_r\}\right)\cup\{\a_i+\cdots+\a_{r+1}\}.\] Then $p'$ is a partition of $\mu$ and $g(p')=p$ as desired.

This establishes that $g$ is a bijection between $P_\Lambda(\mu)$ and $P_{B_r}(\hroot_{B_r})$. Thus $K_\Lambda(\mu)=K_{B_r}(\hroot_{B_r})$.
\end{proof}

\begin{corollary}\label{cor:<1,1>toA}
Consider the Lie algebra of type $A_{r+1}$ with $r\geq 2$, with positive roots $\Phi_{A_{r+1}}^+$. 
If $\Lambda=\Phi_{A_{r+1}}^+\setminus\{\a_{r+1}\}$, then
\[\wp_{\Lambda}(2\hroot_{A_{r+1}} - \alpha_1 - \alpha_{r+1}) = \js(\langle 1, 1 \rangle, \langle 1, 1 \rangle, r).\]
\end{corollary}
\begin{proof}
By Theorem \ref{thm:netheight} we know that  $\delta(\langle 1, 1 \rangle, \langle 1, 1 \rangle, r)= \varepsilon_1+\varepsilon_2-\varepsilon_{r+1}-\varepsilon_{r+2}=2\hroot_{A_{r+1}}-\alpha_1 - \alpha_{r+1}$ and the result follows from Theorem \ref{thm:jugtopart}.
\end{proof}

Finally, Theorem~\ref{thm:solution-problem 1} follows from Lemma \ref{lem:BtoA} and Corollary \ref{cor:<1,1>toA}.

\emph{Proof of Theorem~\ref{thm:solution-problem 2}:} Recall that the highest root of the Lie algebra of type $C_r$ is  $\hroot_{C_r}=2\a_1+2\a_2+\cdots+2\a_{r-1}+\a_r$.
As in the previous section, we establish a bijection from the set of partitions of $\hroot_{C_{r}}$ using the positive roots in $\Phi_{C_r}^+$ to the set of partitions of  $2\hroot_{A_{r}}$ 
using the positive roots in $\Phi_{A_{r}}^+$.
We then give a bijection between the set of partitions of $2\hroot_{A_{r}}$ and the desired  juggling sequences.

\begin{lemma}\label{lem:AtoC}
If $r\geq 3$, then 
$\wp_{A_r}(2\hroot_{A_r})=\wp_{C_r}(\hroot_{C_r}).$
\end{lemma}

\newcommand{\CtoA}{g}
\begin{proof}
We provide a bijection from the sets of partitions $P_{C_r}(\hroot_{C_r})$ and $P_{A_r}(2\hroot_{A_r})$. Any partition $p\in P_{C_r}(\hroot_{C_{r}})$ either contains
\begin{enumerate}
    \item exactly one positive root of the form  $\a_i+\cdots+\a_{j-1}+2\a_j +2\a_{j+1}+\cdots+2\a_{r-1}+\a_{r}$, with $1\leq i<j\leq r-1$ \label{step:c1}, or 
    \item exactly one positive root of the 
     form  $2\a_i+\cdots+2\a_{r-1}+\a_{r}$, with $1\leq i\leq r-1$ \label{step:c2}, or 
    \item only positive roots of the form $\a_i$ for $1\leq i\leq r$ and $\alpha_i + \cdots + \alpha_j$ with $1\leq i\leq j\leq r$, with exactly one 
    such root containing $\a_r$.\label{step:c3}
\end{enumerate}

We now define a map $g:P_{C_r}(\hroot_{C_r})\to P_{A_r}(\hroot_{A_r})$ as follows. If $p\in P_{C_r}(\hroot_{C_r})$ contains the part as in  \eqref{step:c1}, then let 
\[g(p)=\left(p\setminus\{\a_i+\cdots+\a_{j-1}+2\a_j +2\a_{j+1}+\cdots+2\a_{r-1}+\a_{r}\}\right)\cup\{\a_i+\cdots+\a_{r},\a_j+\cdots+\a_{r}\}.\]
If $p\in P_{C_r}(\hroot_{C_r})$ contains a part as in  \eqref{step:c2}, then let 
\[g(p)=\left(p\setminus\{2\a_i+\cdots+2\a_{r-1}+\a_{r}\}\right)\cup\{\a_i+\cdots+\a_{r},\a_i+\cdots+\a_{r}\}.\]
If $p\in P_{C_r}(\hroot_{C_r})$ contains only parts as in  \eqref{step:c3}, then let 
\[g(p)=p\cup\{\a_r\}.\]

By construction, for any $p\in P_{C_r}(\hroot_{C_r})$, $g(p)$ is a partition of $2\hroot_{A_r}$, so the function is injective. We now establish that $g$ is surjective. To see this, take a partition $p=\{\beta_1,\beta_2,\ldots,\beta_l\}\in P_{A_r}(2\hroot_{A_r})$, where $\beta_i\in\Phi_{A_r}^+$ for all $1\leq i\leq l$. Note that either 
\begin{enumerate}
    \item[Case 1:] $\a_r$ appears twice as a part in $p$, or  \label{case:c1} 
    \item[Case 2:] $\a_r$ and $\a_i+\cdots+\a_r$ appear as parts in $p$, with $1\leq i\leq r-1$, or  \label{case:c2}
    \item[Case 3:] $\a_i+\cdots+\a_r$ and $\a_j+\cdots+\a_r$ appear as parts in $p$, with $1\leq i,j\leq r-1$. \label{case:c3} 
\end{enumerate} 

In Case 1 and Case 2, consider the partition $p'=p\setminus\{\a_r\}$ as positive roots in $\Phi_{C_r}^+$. Then $p'$ is a partition of $\hroot_{C_r}$ and $g(p')=p$ as desired.

In Case 3, if $i=j$, then consider the partition \[p'=\left(p\setminus\{\a_i+\cdots+\a_r,\a_i+\cdots+\a_r\}\right)\cup\{2\a_i+\cdots+2\a_{r-1}+\a_r\}\] as positive roots in $\Phi_{C_r}^+$. Then $p'$ is a partition of $\hroot_{C_r}$ and $g(p')=p$ as desired. If $i\ne j$ then without loss of generality assume $i<j$.  Consider the partition \[p'=\left(p\setminus\{\a_i+\cdots+\a_r,\a_j+\cdots+\a_r\}\right)\cup\{\a_i+\cdots+\a_{j-1}+2\a_j+\cdots+2\a_{r-1}+\a_r\},\] positive roots in $\Phi_{C_r}^+$. Then $p'$ is a partition of $\hroot_{C_r}$ and $g(p')=p$ as desired.

This establishes that $g$ is a bijection between $P_{C_r}(\hroot_{C_r})$ and $P_{A_r}(2\hroot_{B_r})$. Thus $K_{C_r}(\hroot_{C_r})=K_{A_r}(2\hroot_{A_r})$, as desired.
\end{proof}

Finally, Theorem~\ref{thm:solution-problem 2} follows from Corollary \ref{cor:hrootAtoMJS} and Lemma \ref{lem:AtoC}.

\emph{Proof of Theorem \ref{thm:solution-new problem}:} Recall that the highest root of the Lie algebra of type $D_r$ is  $\hroot_{D_r}=\a_1+2\a_2+\cdots+2\a_{r-2}+\a_{r-1}+\a_r$.

\begin{lemma}\label{lem:DtoBtoA}
If $r\geq 4$ and $\Lambda=\Phi_{A_{r-1}}^+\setminus\{\a_{r-1}\}$, then \[K_{D_r}(\hroot_{D_r})=5\cdot  K_{B_{r-2}}(\hroot_{B_{r-2}})=5\cdot K_\Lambda(\hroot_{A_{r-1}}-\a_{1}-\a_{r-1}).\]
\end{lemma}
\begin{proof}
The first equality follows by setting $q=1$ in \cite[Theorem 5.2]{HIO} and 
the second equality follows from Lemma \ref{lem:BtoA}.
\end{proof}

Finally, Theorem~\ref{thm:solution-new problem} follows by replacing $r+1$ with $r-1$ in Corollary \ref{cor:<1,1>toA} and from Lemma~\ref{lem:DtoBtoA}.

\section{Applications, connections, and future work}
\label{sec:connections}
Having found a link between juggling sequences and Kostant's partition function we provide some applications and connections between both areas. We provide open questions throughout.

\subsection{New juggling concepts}
\subsubsection{The juggling polytope}
Kostant's partition function can be viewed as counting lattice points of certain polytopes called {\em flow polytopes} \cites{BBCV,TP,MM1}. Because of our main correspondence we can similarly define a polytope that represents juggling sequences.  We refer the reader to \cite{CCD} for background on polytopes.

\begin{definition}[{{\cite[Section 2]{BB}}}] \label{def:flowpoly}
If $\mu$ is a weight of a Lie algebra of type $A_r$, then a {\em type $A_r$ flow} with {\em netflow} $\mu$ is a function $f:\Phi_{A_r}^+ \to \mathbb{R}_{\geq 0}$ such that $\sum_{\beta \in \Phi^+} f(\beta) \beta = \mu$. 
We let $\mathcal{F}_{A_r}(\mu)$ be the set of type $A_r$ flows with netflow $\mu$. The set $\mathcal{F}_{A_r}(\mu)$ forms a polytope called a {\em flow polytope}. 
\end{definition}

The {\em lattice points} of $\mathcal{F}_{A_r}(\mu)$ (the integral type $A_r$ flows with netflow $\mu$) correspond to partitions of the weight $\mu$. Thus, the number of lattice points of $\mathcal{F}_{A_r}(\mu)$ is given by Kostant's partition function $K_{A_r}(\mu)$.

\begin{definition}
A {\em real-valued juggling state} ($\mathbb{R}$-juggling state) is a vector $\bfs=\langle s_1,\ldots,s_h \rangle$ where each $s_i$ is a nonnegative real number representing $s_i$ fragments of balls at height $i$.
\end{definition}

\begin{definition}
A {\em real-valued juggling sequence} ($\mathbb{R}$-juggling sequence) is a sequence of $\mathbb{R}$-juggling states $S=( \bfs_0,\bfs_1,\ldots,\bfs_n)$ where successive states $\bfs_{i-1} = \langle s_1,\ldots,s_h\rangle$ and $\bfs_i$ satisfy 
\[
\bfs_{i} =  \langle s_2+b_1,\ldots,s_h+b_{h-1}, b_h,\ldots,b_{h'}\rangle,
\]
where $(b_1,\ldots,b_{h'})$ is a tuple of nonnegative real numbers satisfying $\sum_{i=1}^{h'} b_i = s_1$.
\end{definition}

We interpret $\mathbb{R}$-juggling sequences  as redistributing fragments of balls that were at height one into fragments of balls at different heights. Figure \ref{fig:real juggling example} illustrates the real-valued juggling sequence 
$S=( 
\langle 1 \rangle,
\langle 1/2,1/4,1/4\rangle,
\langle 1/2,1/2\rangle,
\langle 1\rangle)$.
\begin{figure}[htbp]
    \centering
    \includegraphics[scale= 2.5]{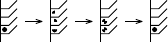}
    \caption{An $\mathbb{R}$-juggling sequence.}
    \label{fig:real juggling example}
\end{figure}

Let $\RJS(\bfa, \bfb,m,n)$ be the set of  $\mathbb{R}$-juggling sequences with initial state $\bfa$, terminal state $\bfb$, hand capacity $m$, and length $n$. Note that $\JS(\bfa, \bfb,n,m)$ is the set of integral juggling sequences in $\RJS(\bfa, \bfb, n,m)$. As before, if there is no hand capacity constraint we omit $m$. The next result shows that $\RJS(\langle \mu_1,\ldots,\mu_r\rangle, \langle \mu_1+\cdots + \mu_r\rangle,r)$ is integrally equivalent (see \cite[\S 2]{MMS}) to a flow polytope.

\begin{proposition} \label{prop:rjs is a polytope}
Let $\mu=\langle \mu_1,\ldots,\mu_{r+1}\rangle$ be a weight of a Lie algebra of type $A_r$.  The set $\RJS(\langle \mu_1,\ldots,\mu_r\rangle, \langle \mu_1+\cdots + \mu_r \rangle,r)$ is a convex polytope integrally equivalent to the flow polytope $\mathcal{F}_{A_r}(\mu)$.
\end{proposition}

\begin{proof}
We extend the bijection $\Gamma$ from Theorem~\ref{thm:prop:cor} to a bijection
\[
\Gamma: \mathcal{F}_{A_r}(\mu) \to \RJS(\langle \mu_1,\ldots,\mu_r\rangle, \langle \mu_1+\cdots + \mu_r\rangle,r),
\]
as follows. Given a flow $f$ in $\mathcal{F}_{A_r}(\mu)$, we let $R:=\Gamma(f)$ be the $\mathbb{R}$-juggling sequence $R=(\bfs_0,\hdots,\bfs_r)$ defined similarly to Equation~\eqref{map:partition2jugglingsequence} as
\begin{center}    
\resizebox{\textwidth}{!}{
$
    \bfs_{i}=\begin{cases}\langle \mu_1,\mu_2,\mu_3,\ldots, \mu_{r}\rangle&\mbox{if $i=0$}\\\langle \mu_{i+1},\mu_{i+2},\ldots, \mu_{r}\rangle+\left\langle\displaystyle\sum_{j=1}^i f(\e_j-\e_{i+1}),\displaystyle\sum_{j=1}^i f(\e_j-\e_{i+2}),\ldots,\displaystyle\sum_{j=1}^i f(\e_j-\e_{r+1})\right\rangle&\mbox{if $1\leq i\leq r$.}
    \end{cases}
$
}
\end{center}
As in the proof of Theorem~\ref{thm:prop:cor}, we have that $\Gamma(f) \in \RJS(\langle \mu_1,\ldots,\mu_r\rangle, \langle \mu_1+\cdots+\mu_r\rangle,r)$ and that the inverse function is well defined. That is, $\Gamma^{-1}$ is the map $R\mapsto f$ where $f(\e_j-\e_k) = ({\bfs_j})_{k-j}-({\bfs}_{j-1})_{k-j+1}$, where $({\bfs_j})_{i}$ denotes the $i$th entry of the juggling sequence ${\bfs_j}$.
The map $\Gamma$ is an affine transformation that is a bijection between $\mathcal{F}_{A_r}(\mu)$ and $\RJS(\langle \mu_1,\ldots,\mu_r\rangle, \langle \mu_1+\cdots + \mu_r\rangle, r)$ and preserves their lattices. 
\end{proof}

Because of this result, the set  $\RJS(\langle \mu_1,\ldots,\mu_r \rangle, \langle \mu_1+\cdots + \mu_r\rangle,r)$ will be called a {\em juggling polytope}. By standard properties of polytopes we obtain the following polynomiality result for juggling sequences.

\begin{corollary}\label{cor:polynomiality}
Let $\mu_i \in \mathbb{N}$ for $i=1,\ldots,r$. The function $\js(\langle t\mu_1,\ldots,t\mu_r\rangle, \langle t(\mu_1+\cdots + \mu_r)\rangle,r)$, defined for positive integers $t$,  is a polynomial in $t$ with degree and leading term equal to the dimension and volume of the juggling polytope $\RJS(\langle \mu_1,\ldots,\mu_r\rangle, \langle \mu_1+\cdots + \mu_r\rangle, r)$.
\end{corollary}

\begin{proof}
If $\mu_i \in \mathbb{N}$ for $i=1,\ldots,r$, then $\mu=\sum_{i=1}^{r}\mu_i\e_i-(\sum_{i=1}^{r}\mu_i)\e_{r+1}$ is a weight of the Lie algebra of type $A_r$. Moreover, the polytope \[\mathcal{F}_{A_r}(\mu)=\RJS(\langle \mu_1,\ldots,\mu_r\rangle, \langle \mu_1+\cdots + \mu_r\rangle,r)\] has integral vertices \cite[Lemma 2.1]{Hille}, \cite[Prop. 2.5]{MM2} and the function \[K_{A_r}(t \mu)=\js(\langle t\mu_1,\ldots,t\mu_r\rangle, \langle t(\mu_1+\cdots + \mu_r)\rangle,r)\] is the Ehrhart polynomial of this polytope. The degree and leading term of this polynomial are given by the dimension and volume of the polytope, respectively \cite[Chapter 3]{CCD}. 
\end{proof}

Moreover, from the theory of flow polytopes there are formulas for $K_{A_r}(\mu)$ for $\mu_i \in \mathbb{Z}_{\geq 0}$ as a weighted sum of values of Kostant's partition function at weights independent of $\mu$ (see Section~\ref{sec:sb}). These formulas are due to Lidskii \cite{lidskii} and generalized to $K_{\Lambda}(\mu)$ by Baldoni--Vergne \cite{BB} and proved via polytope subdivisions in \cites{MM2,KMS}.  

We let $\multiset{n}{k} := \binom{n+k-1}{k}$.  For weak compositions ${\bf j}=(j_1,\ldots,j_n)$ and ${\bf i}=(i_1,\ldots,i_n)$ we say that ${\bf j}$ {\em dominates} ${\bf i}$ if $\sum_{i=1}^k j_k \geq \sum_{i=1}^k i_k$ for $k=1,\ldots,n$. Also, given a weight $\mu=\sum_{i=1}^{r+1}\mu_i\e_i$ of a Lie algebra of type $A_r$, in what follows we let $K(\mu_1,\mu_2,\ldots,\mu_{r+1})=K(\mu)$.

\begin{theorem}[{\cite[Propositions~39 and 34]{BB}}]
For a weight $\mu=\sum_{i=1}^{r+1}\mu_i\e_i$ with $\mu_i \in \mathbb{Z}_{\geq 0}$ for all $1\leq i\leq r$, we have that 
\begin{align*}
&K_{A_r}(\mu)\\
&= 
\sum_{{\bfj}} \binom{\mu_1+r-1}{j_1}\binom{\mu_2 +r-2}{j_2} \cdots \binom{\mu_{r-1}}{j_{r-1}} K_{A_{r-2}}(j_1-r+1,j_2-r+2,\ldots,j_{r-2}-2,  j_{r-1}-1), \\ 
&=\sum_{{\bfj}} \multiset{\mu_1+1}{j_1}\multiset{\mu_2}{j_2} \cdots \multiset{\mu_{r-1}+3-r}{j_{r-1}}  K_{A_{r-2}}( j_1-r+1,j_2-r+2,\ldots,j_{r-2}-2,  j_{r-1}-1),
\end{align*}
where both sums are over weak compositions  $\bfj=(j_1,\ldots,j_{r-1})$ of $\binom{r}{2}$ that dominate the composition $(r-1,r-2,\ldots,1)$..
\end{theorem}

We translate this result to juggling sequences using Theorem~\ref{thm:unresA}. 

\begin{corollary} \label{cor:lidskii-juggling}
For nonnegative integers $\mu_1,\ldots,\mu_r$ we have  
\begin{multline*}
    \js(\langle \mu_1,\ldots,\mu_r\rangle, \langle \mu_1+\cdots+\mu_r\rangle,r) \\
    =\sum_{{\bf j}} \binom{\mu_1+r-1}{j_1}\binom{\mu_2 +r-2}{j_2} \cdots \binom{\mu_{r-1}}{j_{r-1}} \js( \langle j_1-r+1,j_2-r+2,\ldots,j_{r-2}-2 \rangle, \langle 1-j_{r-1} \rangle, r-2),\\
= \sum_{{\bf j}} \multiset{\mu_1+1}{j_1}\multiset{\mu_2}{j_2} \cdots \multiset{\mu_{r-1}+3-r}{j_{r-1}}  \js( \langle j_1-r+1,j_2-r+2,\ldots,j_{r-2}-2 \rangle, \langle 1-j_{r-1} \rangle, r-2),
\end{multline*}
where both sums are over weak compositions  $\bfj=(j_1,\ldots,j_{r-1})$ of $\binom{r}{2}$ that dominate the composition $(r-1,r-2,\ldots,1)$.
\end{corollary}

In particular, this shows that $\js(\langle \mu_1,\ldots,\mu_r\rangle, \langle \mu_1+\cdots + \mu_r\rangle,r)$ is a polynomial in $\mu_1,\ldots,\mu_r$.

\begin{remark}
Note that if a composition $(j_1,\ldots,j_n)$ dominates a composition $(i_1,\ldots,i_n)$ then $(j_n,\ldots,j_1)$ is dominated by $(i_n,\ldots,i_1)$. Thus in the equation above, if $(j_1,\ldots,j_{n-1})$ dominates $(r-1,r-2,\ldots,1)$ then $j_{n-1} \leq 1$.
\end{remark}

We end this section by discussing two examples of juggling polytopes equivalent to well-known flow polytopes with interesting face structures and volumes.  For more results on volumes and lattice points of flow polytopes, see \cite{MP,MM2,caracol,LMStD,MY,JK}.

 \begin{figure}[h]
    \centering
    \hspace{-0.2in}
    \resizebox{.3\textwidth}{!}{
    \subcaptionbox{$\RJS(\langle 1 \rangle, \langle 1 \rangle,3) = \mathcal{F}_{A_3}(\hroot_{A_3})$}%
  [.35\textwidth]{\includegraphics[scale=0.8]{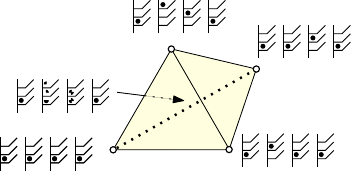}
  } }
  \resizebox{.4\textwidth}{!}{
  \subcaptionbox{$\RJS(\langle 1,1,1 \rangle, \langle 3 \rangle,3) = \mathcal{F}_{A_3}(1,1,1,-3)$}%
  [.4\textwidth]{\includegraphics[scale=0.7]{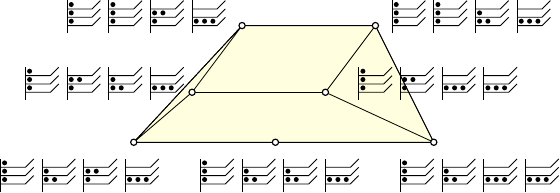}
  }  
  }
  \resizebox{.3\textwidth}{!}{
    \subcaptionbox{$\RJS(\langle 1,1 \rangle, \langle 1,1 \rangle,3)$}%
  [.35\textwidth]{\includegraphics[scale=0.9]{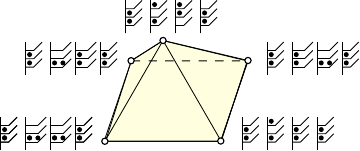}
  } }
   \caption{Examples of juggling polytopes.}
    \label{fig:juggling polytope}
\end{figure}
 
\begin{example} \label{ex:CRY}
The polytope $\RJS(\langle 1 \rangle, \langle 1 \rangle, r) = \mathcal{F}_{A_r}(\tilde{\alpha}_{A_r})$ is known as the Chan--Robbins--Yuen (CRY) polytope \cite{CRY}; see Figure~\ref{fig:juggling polytope}(a). This polytope has $2^{r-1}$ vertices and dimension $\binom{r}{2}$. In unpublished work of Postnikov and Stanley, they applied a result of Baldoni--Vergne \cite{BB} to show that the normalized volume of the polytope equals a value of Kostant's partition function. 
\begin{equation} \label{eq:fundamental-relation}
    {\textstyle \binom{r}{2}!}\cdot \vol\, \mathcal{F}_{A_r}(\tilde{\alpha}_{A_r}) = K_{A_r}({\textstyle \sum_{\beta \in \Phi^+_{A_r}}} \beta).
\end{equation}

We use Proposition~\ref{prop:rjs is a polytope} and Theorem~\ref{thm:prop:cor} to write Equation~\eqref{eq:fundamental-relation} in terms of juggling sequences as
\begin{equation*}
{\textstyle \binom{r}{2}!}\cdot \vol\left(\RJS(\langle 1 \rangle, \langle 1 \rangle, r)\right) = \js(\langle 1,2,\ldots,r-2\rangle, \langle {\textstyle \binom{r-1}{2}}\rangle,r-2).
\end{equation*}
Moreover, the value of Kostant's partition function in the right hand side of Equation~\eqref{eq:fundamental-relation} was computed by Zeilberger \cite{Z} via a variant of a constant term identity of Morris to give a product of Catalan numbers $\Cat_n := \frac{1}{n+1}\binom{2n}{n}$,
\begin{equation} \label{eq:ZprodCat}
    \js(\langle 1,2,\ldots,r-2\rangle, \langle {\textstyle \binom{r-1}{2}} 
\rangle,r-2) =\Cat_1\cdots \Cat_{r-2}.
\end{equation}
\end{example}

\begin{example} \label{ex:Tesler}
The polytope $\RJS(\langle 1,\ldots,1 \rangle, \langle r \rangle, r) = \mathcal{F}_{A_r}(1,\ldots,1,-r)$  is known as the {\em Tesler polytope} \cite{TP}; see Figure~\ref{fig:juggling polytope}(b). This polytope has $r!$ vertices and dimension $\binom{r}{2}$. Using a result of Baldoni--Vergne and a constant term identity, the authors in \cite{TP} showed that the  normalized volume of this polytope equals
\begin{equation}
{\textstyle \binom{r}{2}!}\cdot \vol \left(\RJS(\langle 1,\ldots,1 \rangle, \langle r \rangle, r)\right) = \big\lvert\SYT(r-1,r-2,\ldots,1)\big\rvert \cdot\Cat_1\cdots \Cat_{r-1},
\end{equation}
where $|\SYT(\lambda)|$ is the number of standard Young tableaux of shape $\lambda$. 
\end{example}

\begin{remark}\label{rem:formulas}
The number of lattice points of the CRY polytope is given by $K_{A_r}(\tilde{\alpha}_{A_r})$. For other Lie types, variants of the CRY polytope whose lattice points are given by partitions of the highest root  (see Section~\ref{sec:HIOsubsec}) were defined in \cite{MM1}. These variants also have remarkable volumes proved by Corteel, Kim, and M\'esz\'aros \cite{CKM1,CKM2}: 
\begin{itemize}
    \item $CRY_{C_{r+1}}:=\mathcal{F}_{C_{r+1}}(\tilde{\alpha}_{C_{r+1}})$ has $3^r$ vertices, and volume $2^{r(r-1)}\Cat_1\cdots \Cat_{r-1}$.
    \item $CRY_{D_{r+1}}:=\mathcal{F}_{D_{r+1}}(\tilde{\alpha}_{D_{r+1}})$  has  $3^r-2^r$ vertices, and volume $2^{r(r-2)} \Cat_1\cdots \Cat_{r-1}$.
\end{itemize}
These volume formulas are proved case by case using in part constant term identities. These results suggest that there is a uniform formula or a reduction to the type $A$ case. 
\end{remark}

\subsubsection{The juggling poset}

There is an interesting graded poset defined by Armstrong and further studied by O'Neill \cite{JO} on the configurations counted by Kostant's partition function. In this context, these configurations are referred to as Tesler matrices \cite{AGHRS,H,TP} because of a connection to diagonal harmonics and the poset in \cite{JO} is called the Tesler poset. Since by Theorem~\ref{thm:prop:cor} these configurations are in bijection with juggling sequences, we can define the analogous graded poset for juggling sequences. For background on posets see \cite[Chapter~3]{EC1}.

\begin{definition}
Consider juggling states $\bfa=\langle a_1,\ldots,a_s\rangle$ and $\bfb=\langle b_1,\ldots,b_t\rangle$, satisfying $a_1+\cdots+a_s = b_1+\cdots+b_t$, and positive integers $n$ and $m$. Let $\PJS(\bfa,\bfb,n,m)$ be the poset with elements $\JS(\bfa, \bfb, n,m)$  and cover relations $S \lessdot S'$ if $S'$ is obtained from $S$ by replacing both a throw at time $i$ to height $j$ and a throw at time $i+j$ to height $k$ in $S$ by one throw at time $i$ to height $j+k$. As before, if there is no hand capacity constraint we omit $m$. See Figure~\ref{fig:juggling poset} for examples. 
\end{definition}

\begin{remark}
After applying the bijection $\Gamma$ from Theorem~\ref{thm:jugtopart}, the poset $\PJS(\bfa,\bfb,n,m)$ is equivalent to a poset with elements $Q_{\Lambda}(\delta)$ and cover relations $p \lessdot p'$ if $p'$ is obtained from $p$ by replacing one instance of $\e_i-\e_{i+j}$ and $ \e_{i+j}-\e_{i+j+k}$ in $p$ by $\e_i-\e_{i+j+k} \in \Lambda$.
\end{remark}

In the special case of the correspondence in  Theorem~\ref{thm:prop:cor}, the juggling poset $\PJS(\langle \mu_1,\ldots, \mu_n\rangle, \langle \mu_1 + \cdots + \mu_n \rangle, n)$ is precisely the dual to the Tesler poset in \cite{JO} with elements $P_{A_r}(\mu)$. This poset is graded and has unique minimal and maximal elements \cite[Section 3.2]{JO}. It is not difficult to see that the juggling poset $\PJS(\langle 1 \rangle, \langle 1 \rangle, n,1)$ is isomorphic to the Boolean lattice of rank $n-1$ (see Figure~\ref{fig:juggling poset}(a)). Interestingly, the  juggling posets with binary initial state $\bfa$ and terminal state $\langle a_1+\cdots + a_n\rangle$ have surprisingly simple characteristic polynomials $\chi(\cdot)$. 

\begin{theorem}[\cite{JO}, Theorem 3]\label{th:characteristic poly}
Let $\bfa \in [0,1]^{n}$ then 
\[
\chi(\PJS(\bfa, \langle a_1 + \cdots + a_n \rangle, n),q) = (q-1)^{\sum_{i=1}^n (n-i)a_i}. 
\]
\end{theorem}

In particular, $\chi(\PJS(\langle 1 \rangle, \langle 1 \rangle, n),q) = (q-1)^{n-1}$ for the Boolean lattice and $$\chi(\PJS(\langle 1^n \rangle, \langle n \rangle, n),q) = (q-1)^{\binom{n}{2}}$$ where $\langle 1^n \rangle = \langle 1,\ldots,1\rangle$ (see Figure~\ref{fig:juggling poset}(b) for an example of $\PJS(\langle 1,1,1\rangle, \langle 3 \rangle ,3)$). The latter result of O'Neill was originally conjectured by Armstrong. It would be of interest to further study Tesler posets and the more general juggling posets $\PJS(\bfa,\bfb,n,m)$ (see Figure~\ref{fig:juggling poset}(c)).
\begin{figure}[htbp]
     \subcaptionbox{$\PJS(\langle 1 \rangle, \langle 1 \rangle,3)$}%
  [.25\textwidth]{ \includegraphics[scale=0.5]{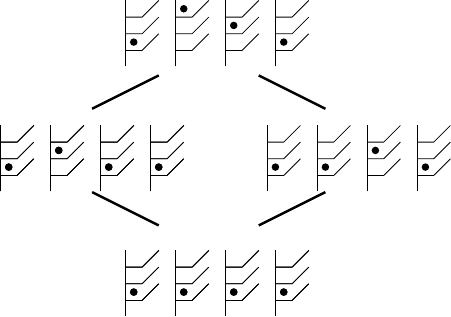}
  }
     \subcaptionbox{$\PJS(\langle 1,1,1 \rangle, \langle 3 \rangle,3)$}%
  [.4\textwidth]{ \includegraphics[scale=0.6]{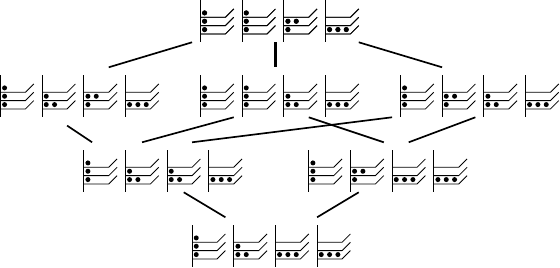}
  }
   \subcaptionbox{$\PJS(\langle 1,1 \rangle, \langle 1,1 \rangle,3)$}%
  [.3\textwidth]{ \includegraphics[scale=0.4]{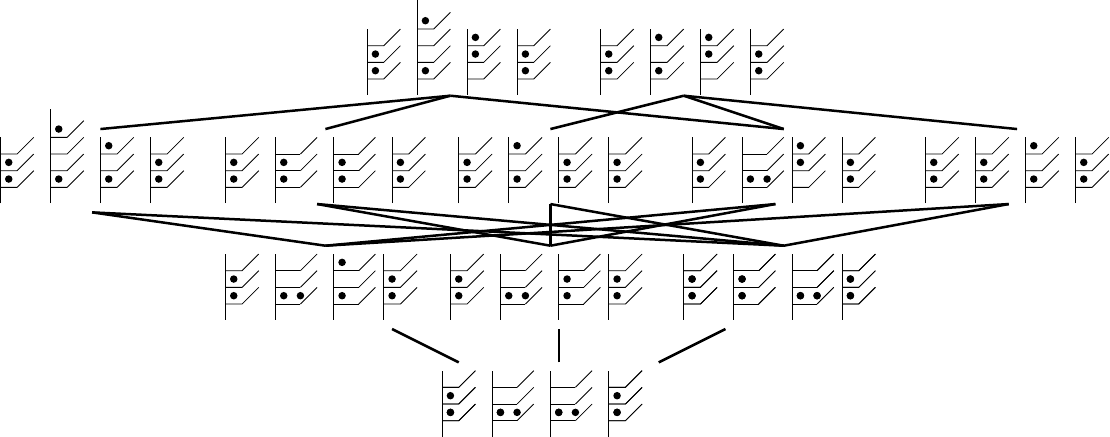}
  }
  \caption{Examples of juggling posets.}
    \label{fig:juggling poset}
\end{figure}

\subsection{Connections}
\subsubsection{Permanent and determinant formulas}
A result of Chung and Graham \cite[Theorem 2]{CG}, extended to multiplex juggling by Butler and Graham \cite{BG}, gives
a method for counting juggling sequences with hand capacity one
by taking permanents. This result can be extended to the case of a restricted set of throws and hand capacity $1$. Moreover, the resulting permanent equals a determinant of a very similar matrix, and thus can be computed efficiently.

\begin{definition}
Given a set of positive roots $\Lambda \subseteq \Phi^+_{A_r}$ with corresponding set of throws $\Gamma(\Lambda)$, let $M(r,\Lambda)=(m_{i,j})$ and $N(r,\Lambda)=(n_{i,j})$  be the $(r+1)\times (r+1)$ matrices with entries
\[
m_{i,j} = \begin{cases}
1 & \text{ if } j \geq i \text{ and } \e_i-\e_{j+1} \in \Lambda\\
1 & \text{ if } j = i-1\\
0 & \text{ otherwise,}
\end{cases}\qquad\mbox{and}\qquad 
n_{i,j} = \begin{cases}
\ph1 & \text{ if } j \geq i \text{ and } \e_i-\e_{j+1} \in \Lambda\\
-1 & \text{ if } j = i-1\\
\ph0 & \text{otherwise.}
\end{cases}.
\]
\end{definition}

\begin{theorem} \label{prop:juggling-perm-det}
Given a set $\Lambda \subseteq \Phi^+_{A_r}$ then
\[
K_{\Lambda}(\tilde{\alpha}_{A_r})=\js_{\Gamma(\Lambda)}(\langle 1 \rangle, \langle 1 \rangle, r)  =  \perm(M(r,\Lambda)) = \det(N(r,\Lambda)).
\]
\end{theorem}

\begin{proof}
The first equality follows by Theorem~\ref{thm:restrictparts} for the weight $\mu = \tilde{\alpha}_{A_r}$. The second equality is essentially the result in \cite[Theorem 2]{CG} (the case $b=1$) allowing for a restricted set of roots $\Lambda$. The last equality was observed by Morales and Wang (unpublished) and can be verified by induction or using \cite[Example 2.2.4]{EC1}.
\end{proof}

\begin{example}
Let $\Lambda$ be the set of positive roots in $\Phi^+_{A_3}$ of length at most two. More precisely,
\[
\Lambda = \{\e_1-\e_2, \e_1-\e_3,\e_2-\e_3,\e_2-\e_4, \e_3-\e_4\} \textup{ and }  \Gamma(\Lambda) = \{T_{1,1}, T_{1,2},T_{2,2},T_{2,3},T_{3,3}\}.
\]
By Theorem~\ref{prop:juggling-perm-det}
\[
K_{\Lambda}(\tilde{\alpha}_{A_3}) = \js_{\Gamma(\Lambda)}(\langle 1 \rangle, \langle 1 \rangle, 3) = \perm \begin{bmatrix}
1 & 1 &0 &0   \\
1 & 1 & 1 &0  \\
0  & 1 & 1 & 1  \\
0  &0 & 1 & 1  
\end{bmatrix} = 
\det \begin{bmatrix}
\ph1 & \ph1 &\ph0 &\ph0   \\
-1 &\ph 1 &\ph 1 &\ph0  \\
\ph0  & -1 &\ph 1 &\ph 1 \\
\ph0  &\ph0 & -1 &\ph 1 
\end{bmatrix} = 5.
\]
\end{example}

In light of Theorem~\ref{prop:juggling-perm-det} we ask:
\begin{question}
For what other hand capacities can one develop a permanent or determinant formula to count the number of juggling sequences? 
\end{question}
Given our results in this paper, an answer to this question would yield permanent or determinant formulas for the value of Kostant's partition function and its restriction.

\subsubsection{Classes of generating functions} \label{sec:genfunctions}
The examples of generating functions for juggling sequences in Table~\ref{tab:BG} and for $a_r:=\js(\langle n \rangle, \langle n \rangle,r)$ (see  Remark~\ref{rem: cry recurrence}) are all rational. However, it is not true that the generating function for the number of juggling sequences of a fixed state and length $n$ are rational, algebraic, or even $D$-finite. See \cite[Ch. 6]{EC2} for background on generating functions.

\begin{proposition}Let $\mu = \sum_i \binom{i+1}{2} \alpha_i$, 
the generating function for the sequence $a_r:=\wp_{A_r}(\mu) = \js(\langle 1, 2,\ldots,r \rangle, \langle {\textstyle \binom{r+1}{2}} \rangle,  r)$
is not $D$-finite, in other words, the sequence is not $P$-recursive.
\end{proposition}

\begin{proof}
By Equation~\eqref{eq:ZprodCat} $a_r= \Cat_1\Cat_2\cdots \Cat_r$, which grows superexponentially since $\Cat_r \sim \frac{4^r}{r^{3/2}\sqrt{\pi}}$. Hence, the generating function $\sum_{r\geq 0} a_r x^r$ cannot be $D$-finite~\cite{MS}.  
\end{proof}

Thus, we pose the following.
\begin{question}
Find conditions on initial and terminal states, ${\bf a}$ and ${\bf b}$ respectively, so that the generating function of magic multiplex juggling sequences $a_r := \js({\bf a}, {\bf b}, r)$ is rational, algebraic,  $D$-finite, etc.   
\end{question}

\subsubsection{Positroids}
Juggling sequences appear in the context of positroids in \cite{KLS}. It can be seen that every juggling sequence $\JS(\bfv,\bfv,n,1)$, where $\bfv\in\{0,1\}^n$ is such that $\sum_{i=i}^n {v_i}=k$, corresponds to a positroid of rank $k$ seen as a bounded affine permutation (see \cite[Section 3]{KLS}).

\begin{question}
What further connections exist between positroids and juggling sequences?
\end{question}

\subsubsection{Weight multiplicities}
Kostant's weight multiplicity gives the multiplicity of the weight $\mu$ in the highest weight representation $L(\lambda)$ of a simple Lie algebra $\mathfrak{g}$ \cite{KMF}. This formula is given by 
\begin{align}
    m(\lambda,\mu)=\sum_{\sigma\in W}(-1)^{\ell(\sigma)}K(\sigma(\lambda+\rho)-\mu-\rho),\label{eq:KWMF}
\end{align}
where $W$ denotes the Weyl group of $\mathfrak{g}$, $\ell(\sigma)$ denotes the length of a Weyl group element, $K$ denotes Kostant's partition function, and $\rho=\frac{1}{2}\sum_{\a\in \Phi^+}\a$.

Much work has been done in giving closed formulas for Equation~\eqref{eq:KWMF}. This includes determining the support of the function when $\lambda$ is the highest root of a Lie algebra and $\mu$ is zero or a positive root \cite{CHI,PH,PHThesis,HIW}, determining solutions to associated $q$-analog problems \cite{HIO,HIS,HarrisLauber}, and providing visualizations for the support of \eqref{eq:KWMF} in low rank examples \cite{HLM,HLRRRKTDU}. In light of the main results in the current, we pose the following.

\begin{question}
Can one exploit the connection between juggling sequences and Kostant's partition function to provide new formulas for Kostant's weight multiplicity formula?
\end{question}

\section*{Acknowledgments}

A preliminary version of these results were presented in \cite{Simpson}.
This work was supported by the American Institute of Mathematics through their SQuaRE program.  We are very appreciative of their support and funding which made this research collaboration possible. We thank Rafael Gonz\'alez D'Le\'on and Martha Yip for fruitful discussions throughout the process. We also thank Igor Pak and Mark Wilson for their helpful suggestions in Section~\ref{sec:genfunctions} as well as the anonymous referees. We thank Laura Colmenarejo and Viviene Do for correcting the poset in Figure~\ref{fig:juggling poset}(c). C. Benedetti also thanks grant FAPA of the Faculty of Science at Universidad de los Andes. A. H. Morales received support from NSF Grant DMS-1855536.

\section*{Dictionary of Notation}

\begin{longtable}{lll}
$h$ &--& height of a juggling state\\
$m$ &--& hand capacity of a juggling sequence\\
$n$ &--& length of a juggling sequence\\
$p$ &--& multiset of positive roots\\
$r$ &--& rank of Lie algebra\\
$s_i$ &--& entry of juggling sequence vector\\
$A_r$ &--& Lie algebra $\mathfrak{sl}_{r+1}(\mathbb{C})$\\
$B_r$ &--& Lie algebra $\mathfrak{so}_{2r+1}(\mathbb{C})$\\
$C_r$ &--& Lie algebra $\mathfrak{sp}_{2r}(\mathbb{C})$\\
$D_r$ &--& Lie algebra $\mathfrak{so}_{2r}(\mathbb{C})$\\
$K(\mu)$ &--& number of partitions of $\mu$ using positive roots of type $A$\\
$K_\Lambda(\mu)$ &--& number of partitions of $\mu$ using roots from $\Lambda$\\
$P(\mu)$ &--& set of partitions of $\mu$ \\
$P_\Lambda(\mu)$ &--& set of partitions of $\mu$ using roots from $\Lambda$\\
$Q_\Lambda(\mu)$ &--& subset of $P_\Lambda(\mu)$ subject to a hand capacity constraint\\
$S$ &--& juggling sequence\\
$T_{i,j}$ &--& a throw at time $i$ to height $j$\\
$\T$ &--& a set of throws\\
$\T_r$ &--& set of all throws that land by time $r+1$\\
$\bfa$ &--& initial state of a juggling sequence\\
$\bfb$ &--& terminal state of a juggling sequence\\
$\bfs$ &--& juggling state\\
$\bft$ &--& reflected juggling state\\
$\alpha_i$ &--& simple roots of a Lie algebra\\
$\tilde\alpha$ &--& highest root of a Lie algebra\\
$\beta_i$ &--& a positive root\\
$\delta(S)$ &--& net change vector of $S$\\
$\e_i$ &--& standard basis vectors\\
$\mu$ &--& weight of a Lie algebra\\
$\Gamma$ &--& function between (multisets of) positive roots and (multisets of) throws\\
$\Delta$ &--& simple roots of a Lie algebra\\
$\Phi$ &--& root system of a Lie algebra\\
$\Phi^+$ &--& positive roots of a Lie algebra\\
$\Lambda$ &--& subset of $\Phi^+$\\
$\JS$ &--& set of juggling sequences\\
$\js$ &--& number of juggling sequences\\
$\JS_\T$ &--& set of juggling sequences using throws from $\T$\\
$\js_\T$ &--& number of juggling sequences using throws from $\T$\\
$\LJS$ &--& set of labeled juggling sequences\\
$\ljs$ &--& number of labeled juggling sequences\\
$\PJS$ &--& juggling poset\\
$\RJS$ &--& juggling polytope / set of real-valued juggling sequences\\
$(\cdot,\cdot,\cdot)$ &--& parentheses denote a juggling sequence\\
$\langle \cdot,\cdot,\cdot\rangle$ &--& angle brackets denote a juggling state\\
$[\cdot,\cdot,\cdot]$ &--& square brackets denote the components of a labeled juggling state\\
\end{longtable}

\addresseshere
\end{document}